%% file: ochw.tex
\documentclass{amsart}

\usepackage{amsmath,amsthm, latexsym, amssymb,mathrsfs,  xypic}
\usepackage{stackengine,scalerel, url}
\usepackage[colorlinks]{hyperref}

\input{macros}

\newcommand{\locan}{\mathrm{loc-an}}

\newcommand{\calF}{\mathcal{F}}

\newcommand{\calO}{\mathcal{O}}

\newcommand{\calV}{\mathcal{V}}
\newcommand{\wh}[1]{\widehat{#1}}

\newcommand{\colim}{\mathrm{colim}}

\newcommand{\bbA}{\mathbb{A}}
\newcommand{\bbF}{\mathbb{F}}
\newcommand{\bbG}{\mathbb{G}}

\newcommand{\bbP}{\mathbb{P}}
\newcommand{\bbQ}{\mathbb{Q}}
\newcommand{\bbR}{\mathbb{R}}

\newcommand{\bbZ}{\mathbb{Z}}
\newcommand{\bbC}{\mathbb{C}}

\newcommand{\calI}{\mathcal{I}}
\newcommand{\bs}{\backslash}
\newcommand{\loc}{\mathrm{loc}}
\newcommand{\frakp}{\mathfrak{p}}
\newcommand{\diag}{\mathrm{diag}}
\newcommand{\gl}{\mathfrak{gl}}
\newcommand{\calT}{\mathcal{T}}

\input{theoremnumbering}

\begin{document}

\title[OCMF are highest weight vectors in completed cohomology]{Overconvergent modular forms are highest weight vectors in the Hodge-Tate weight zero part of completed cohomology}
\author{Sean Howe}
 
\maketitle
    
\begin{abstract}We construct a  $(\mathfrak{gl}_2, B(\mathbb{Q}_p))$ and Hecke-equivariant cup product pairing between overconvergent modular forms and the local cohomology at $0$ of a sheaf on $\mathbb{P}^1$, landing in the compactly supported completed $\mathbb{C}_p$-cohomology of the modular curve. The local cohomology group is a highest-weight Verma module, and the cup product is non-trivial on a highest weight vector for any overconvergent modular form of infinitesimal weight not equal to $1$.  For classical weight $k\geq 2$, the Verma has an algebraic quotient $H^1(\mathbb{P}^1, \mathcal{O}(-k))$, and on classical forms the pairing factors through this quotient, giving a geometric description of ``half" of the locally algebraic vectors in completed cohomology; the other half is described by a pairing with the roles of $H^1$ and $H^0$ reversed between the modular curve and $\mathbb{P}^1$. Under minor assumptions, we deduce a conjecture of Gouvea on the Hodge-Tate-Sen weights of Galois representations attached to overconvergent modular forms. Our main results are essentially a strict subset of those obtained independently by Lue Pan, but the perspective here is different and the proofs are short and use simple tools: a Mayer-Vietoris cover, a cup product, and a boundary map in group cohomology. \end{abstract}

\setcounter{tocdepth}{1}
\tableofcontents

\section{Introduction}

In this work, we show that cuspidal overconvergent modular forms of infinitesimal weight $\neq 1$ give rise via an explicit construction to highest weight vectors in the compactly supported completed cohomology of the modular curve (Theorem \ref{maintheorem:ocmf-cup} below). Using this result, we compute the Hodge-Tate-Sen weights of the Galois representation attached to an overconvergent eigenform (possibly of infinite slope!) outside of weight 1 and assuming the residual representation is absolutely irreducible (Corollary \ref{c:hts-weights} below). This verifies a conjecture of Gouvea \cite[Conjecture 4]{gouvea:continuity} in most cases.  Our result mirrors the classical picture, where complex cuspidal modular forms are naturally identified with highest weight vectors in the corresponding automorphic representation of $\GL_2(\bbR)$. 

In fact, just as in the complex case, what one obtains more canonically is a map from the associated Verma module: Our maps are most naturally formulated as cup products between overconvergent modular forms and local cohomology groups on $\bbP^1$, giving a direct connection between the structure of completed cohomology and classical geometric representation theory. 

After preparing an earlier draft, we learned of a preprint by Lue Pan \cite{pan:locally-an} giving a complete description of the Hodge-Tate-Sen decomposition of the highest weight vectors in the locally analytic part of completed cohomology. In particular, \cite[Theorem 5.4.2]{pan:locally-an} essentially subsumes our Theorem \ref{maintheorem:ocmf-cup}, and goes much further. There are two advantages of the approach we present here:
\begin{enumerate}
\item Our proof is brief and relatively straightforward; a reader interested in the results of \cite{pan:locally-an} might stop here first to get a quick geometric perspective on why overconvergent modular forms should relate to locally analytic vectors in completed cohomology in the first place.
\item By swapping the role of $H^0$ and $H^1$ between the modular curve and $\bbP^1$, it is easy to see that there is a dual geometric picture connecting to work of Boxer-Pilloni \cite{boxer-pilloni:higher-hida-modular} on higher Hida theory (cf. Remark \ref{remark:dual-picture} below for an explanation of how this may fit into the Shimura isomorphism of \cite{pan:locally-an}). 
\end{enumerate}

\subsection{Summary of construction and results}\label{ss:summary}
The idea of our construction can be stated very naively using the geometry of the Hodge-Tate period map. We fix a prime-to-$p$ level $K^p$ (a compact open subgroup of $\GL_2(\bbA_f^{(p)})$) and write:
\begin{enumerate}
\item $X/\bbC_p$ for the perfectoid compactified modular curve of prime-to-$p$ level $K^p$ and infinite level at $p$,
\item $\calI$ for the ideal sheaf of the boundary (cusps) on $X$,
\item $\pi_\HT: X \rightarrow \bbP^1$ for the Hodge-Tate period map, 
\item $0=[0:1] \in \bbP^1(\bbQ_p)$, 
\item $B \subset \GL_{2,\bbZ_p}$ for the upper triangular Borel stabilizing ${0\in \bbP^1(\bbQ_p)=\bbP^1(\bbZ_p)}$, $N \subset B$ for its unipotent radical and $N_0=N(\bbZ_p)$, 
\item and $z=x/y$ for the canonical coordinate on $\bbP^1$ at $0$.
\end{enumerate} 

The topological closure of the canonical component of the ordinary locus in $X$ is the fiber $X_{\{0\}}$, and any cuspidal overconvergent modular form can naturally be identified with a function in $H^0(X_{|z| \leq \epsilon}, \calI)$ for some $\epsilon$ depending on the radius of overconvergence (as in  \cite{chojecki-hansen-johannson:ocmf, howe:thesis}).  

On the other hand, a result of Scholze \cite{scholze:torsion} implies that the compactly supported $\bbC_p$-completed cohomology $\tilde{H}^1_{c, \bbC_p}$ of the tower of modular curves can be computed as the analytic cohomology $H^1(X, \calI)$. If we consider the cover of $X$ by $X_{|z| \leq \epsilon}$ and $X_{|z| \geq \epsilon}$, then both of these and their intersection $X_{|z|=\epsilon}$ are affinoid perfectoids, and thus the analytic cohomology $H^1(X, \calI)$ is computed by the \v{C}ech/Mayer-Vietoris complex for this covering. As a consequence, we find classes in $H^1(X,\calI)$ are represented by \emph{functions} in $H^0(X_{|z|=\epsilon}, \calI)$.

To obtain the class attached to an overconvergent modular form we do the simplest possible thing that is not obviously trivial: we send the function $f$ in $H^0(X_{|z| \leq \epsilon}, \calI)$ attached to an overconvergent modular form of weight $\kappa$ to the class $[\frac{f}{z}]$. A simple computation verifies that this is a highest weight vector of weight $\Lie(\kappa)-2$; the key point then is to verify that it is not zero! This is accomplished by composing with a certain restriction map to functions on the generic fiber of the Igusa tower $X_{0}^\circ/N_0$ -- this composition can computed explicitly and identified with multiplication by $1 - \Lie(\kappa)$ (which explains why weight one is excluded from our results!). 

When the modular form is defined over a finite extension of $\bbQ_p$, then the whole construction can be carried out over that extension, and we find that the resulting vector lands in the Hodge-Tate weight zero part of completed cohomology. From this we deduce that the Galois representation attached to \emph{any}\footnote{Under the assumption that the residual representation is absolutely irreducible.} overconvergent modular form of weight not equal to 1 defined over $\overline{\bbQ_p}$ admits zero as a Hodge-Tate-Sen weight (Corollary \ref{c:hts-weights}), proving a conjecture of Gouvea previously known only for classical modular forms and finite slope overconvergent modular forms. 

Restricting from overconvergent to classical modular forms, we find our construction gives a geometric realization of the Hodge-Tate weight zero part of the locally algebraic vectors in completed cohomology via a simple cup product between classical modular forms and the Borel-Weil-Bott realizations of algebraic representations in the first cohomology of line bundles on the flag variety $\bbP^1$. In fact, comparing with work of Faltings and Emerton, we find that this describes ``half" of the locally algebraic vectors; the other half come from swapping the role of $H^1$ and $H^0$ between the modular curve and $\bbP^1$.

\subsection{Interpolation of cup products and the main result}
The explicit description in terms of functions summarized above is useful for computations but at first glance appears rather ad hoc. To understand the situation more clearly, we identify our map with a cup product.

The modular sheaf $\omega$ on $X$ is naturally identified with $\pi_\HT^*\calO(1)$, thus we obtain cup product maps 
\begin{equation}\label{eqn:alg-cup-prod-1} H^0(X, \omega^k \otimes \calI) \otimes H^1(\bbP^1, \calO(-k)) \rightarrow H^1(X, \calI). \end{equation}
The right-hand term in the pairing is an algebraic representation of $\GL_2(\bbQ_p)$ which is non-zero for $k \geq 2$. If we pass to smooth vectors for the $\GL_2(\bbQ_p)$-action on the left we obtain the space of classical weight $k$ cusp forms $S_k^\cl$. Pairing a classical form with a distinguished highest weight vector in the algebraic representation gives exactly the class $[f/z]$ described in the summary above: indeed, this highest weight vector is represented by the meromorphic section $y^{-k}/z$ on $\bbP^1$, and the section $y$ of $\calO(1)$ pulls back to the canonical trivialization of $\omega$ via $\pi_\HT$. Theorem \ref{maintheorem:ocmf-cup} below will show this pairing is injective. 

The group $H^1(\bbP^1, \calO(-k))$ admits a natural surjection from the algebraic local cohomology of $\calO(-k)$ at $0$, $H^1_{\{0\},\mathrm{alg}}(\bbP^1, \calO(-k))$ -- representation-theoretically this can be identified with the surjection from the corresponding highest weight Verma module to the algebraic representation. These local cohomology classes can be paired with overconvergent modular forms, giving an extension of (\ref{eqn:alg-cup-prod-1}):  When we pair with a classical modular form, the map factors through the algebraic representation, but in general it does not. Moreover, while the algebraic representations cannot be interpolated outside of classical weights, the local cohomology Verma modules can be interpolated geometrically along with the cup product pairing with overconvergent modular forms. We explain this now. 

\subsubsection{Overconvergent and classical modular forms} For $\kappa$ a continuous $\bbC_p$-valued character of $\bbZ_p^\times$, we define a $B(\bbQ_p)$-equivariant sheaf $\calO(\kappa)$ on the germ of $ 0 =[0:1] \in \bbP^1.$
We write $\omega^\kappa = \pi_{\HT}^{-1} \calO(\kappa)$. When $\kappa$ is the classical character $z \mapsto z^k$, $\calO(\kappa)$ extends to the standard $\GL_2(\bbQ_p)$-equivariant sheaf $\calO(k)$ on $\bbP^1$. 

We consider the space of overconvergent sections
\[ H^{0, \dagger}(X_{\{0\}}, \omega^\kappa \otimes \calI ) = \colim_{\epsilon \rightarrow 0} H^0(X_{|z| \leq \epsilon}, \omega^\kappa \otimes \calI ).  \]
It admits an action of $B(\bbQ_p)$. Moreover, because each neighborhood $|z| \leq \epsilon$ is stabilized by some $\Gamma_0(p^n)$, the action of $B(\bbZ_p)$ on any section extends, and it makes sense to define the subspace of smooth vectors, denoted with a superscript $\sm$, as those stabilized by some compact open subgroup of $\GL_2(\bbQ_p)$. We write 
\[ S_{\kappa}^\dagger = H^{0, \dagger}(X_{\{0\}}, \omega^\kappa \otimes \calI)^\sm. \]
We consider $S_{\kappa}^\dagger$ as a $(\gl_2, B(\bbQ_p))$-module with trivial $\gl_2$-action. For $k\in\bbZ$, restriction gives a natural $(\gl_2, B(\bbQ_p))$-equivariant injection 
\begin{equation}\label{eqn:classical-embedding} S_k^\cl \hookrightarrow S_{k}^\dagger. \end{equation}

\begin{remark}
The traditional definition of overconvergent cusp forms of weight $\kappa$ coincides with the invariants $(S_\kappa^\dagger)^{B(\bbZ_p)}$. We prefer the larger space $S_{\kappa}^\dagger$ here as it is better suited for representation-theoretic arguments. In particular, it is convenient to have the map (\ref{eqn:classical-embedding}) available rather than twisting the weight to see all classical forms. In that vein, we note that if $\chi$ is a finite order character of $\bbZ_p^\times$ then we have a canonical identification of $(\gl_2, B(\bbQ_p))$ representations $S_{\chi\kappa}^\dagger= S_{\kappa}^\dagger(\chi^{-1}),$ where the notation $(\chi^{-1})$ on the right denotes a twist by
\[ \begin{bmatrix} a & b \\ 0 & d \end{bmatrix} \mapsto \chi^{-1}(d|d|_p). \]
In particular, when $\Lie \kappa = k \in \bbZ$, we usually just consider $\kappa=z^k$. 
\end{remark}

\subsubsection{The pairing}
To compute $\tilde{H}^1_{c,\bbC_p}=H^1(X, \calI)$, we take the colimit over $\epsilon$ of the Mayer-Vietoris sequences\footnote{We verify the necessary vanishing in Lemma \ref{lemma:vanishing} so that one can compute the cohomology of $\calI$ using these covers by quasi-Stein perfectoids rather than affinoid perfectoids, though one could also arrange the entire argument using only covers by affinoid perfectoids.} for the covers of $X$ by $X_{|z| \leq \epsilon}$ and $X_{|z|>0}$:
\begin{equation}\label{eqn:intro-fund-ses} 0 \rightarrow \substack{ H^{0,\dagger}(X_{\{0\}}, \calI)  \\ \oplus H^0(X_{|z|>0}, \calI) } \rightarrow \colim_\epsilon H^0(X_{0<|z| \leq \epsilon}, \calI) \rightarrow H^1(X, \calI) \rightarrow 0. \end{equation}

\begin{remark} Merging two terms gives the local cohomology sequence
\[ 0 \rightarrow  H^0(X_{|z|>0}, \calI) \rightarrow H^1_{X_{\{0\}}}(X, \calI) \rightarrow H^1(X, \calI) \rightarrow 0. \]
\end{remark}

We write $H^1_{\{0\}}(\bbP^1, \calO(\kappa^{-1}))$ for the analytic local cohomology of $\calO(\kappa^{-1})$ at $0 \in \bbP^1$; its elements are represented by sections in a punctured neighborhood of $0$ (see \ref{subsub.local-cohomology}). There is a natural identification $\calO(\kappa)^{*}=\calO(\kappa^{-1})$, thus we obtain a pairing
\begin{equation}\label{eq:first-version-ocmf-pairing} S_\kappa^\dagger \otimes H^1_{\{0\}}(\bbP^1, \calO(\kappa^{-1})) \rightarrow H^1(X, \calI)=\tilde{H}^1_{c,\bbC_p} \end{equation}
by pairing an overconvergent modular form with a section on a punctured neighborhood to obtain an element in $\colim_\epsilon H^0(X_{0<z \leq \epsilon}, \calI)$, then mapping to $H^1(X, \calI)$ where the result becomes well-defined. 

We show that (\ref{eq:first-version-ocmf-pairing}) factors through a map of $(\gl_2, B(\bbQ_p))$-modules to the locally-analytic vectors $\tilde{H}^{1,\loc-an}_{c,\bbC_p}$. The space $H^1_{\{0\}}(\bbP^1, \calO(\kappa^{-1}))$ contains a canonical Verma module $V_{\kappa^{-1}}$ of highest weight $\Lie \kappa-2$ spanned by germs of sections meromorphic at $0$, and to describe (\ref{eq:first-version-ocmf-pairing}) further it is convenient to consider the restriction to $V_{\kappa^{-1}},$
\begin{equation}\label{eq:ocmf-verma-pairing} S_\kappa^\dagger \otimes V_{\kappa^{-1}} \rightarrow \tilde{H}^{1,\locan}_{c,\bbC_p}. \end{equation} 

By the equivariance properties of $\pi_\HT$, the prime-to-$p$ action on $V_{\kappa^{-1}}$ is trivial and thus the prime-to-$p$ action on the left is concentrated in $S_\kappa^\dagger$. In particular, for any prime-to-$p$ Hecke eigensystem $\frakp$ we obtain an induced map on eigenspaces
\[ S_{\kappa}^\dagger[\frakp] \otimes  V_{\kappa^{-1}} \rightarrow \tilde{H}^1_{c,\bbC_p}[\frakp]. \]

$V_{\kappa^{-1}}$ is irreducible unless $\Lie \kappa = k \geq 2$, when $V_{\kappa^{-1}}$ admits an algebraic quotient. We write the kernel as $V_{\kappa^{-1}}'$, a Verma module of highest weight $-k$. 

\begin{maintheorem} \label{maintheorem:ocmf-cup}
If $\Lie \kappa \not \in \bbZ_{\geq 1}$, then (\ref{eq:ocmf-verma-pairing}) is a $(\gl_2, B(\bbQ_p))$ and prime-to-$p$ Hecke equivariant injection. For $\kappa=z^k$, $k \geq 2$, the kernel is $S_k^\cl \otimes V_{-k}'$, and the pairing
\[ S_k^\cl \otimes H^1(\bbP^1, \calO(-k)) \hookrightarrow \tilde{H}^{1,\locan}_{c,\bbC_p} \]
induced by passing to the quotient is identified with the restriction of the global $\GL_2(\bbQ_p)$ and prime-to-$p$ Hecke equivariant cup product
\[ H^0(X, \omega^k \otimes \calI) \otimes H^1(\bbP^1, \calO(-k)) \rightarrow H^1(X, \calI). \]
\end{maintheorem}

\begin{remark} If $\Lie \kappa = 1$, then pairing with a \emph{classical} form is easily seen to give zero (see the second part of Lemma \ref{lemma:cup-classical-zero}). For non-classical weight one forms, we expect that the pairing can be non-zero, e.g. for a non-classical weight one specialization of a Hida family. 
\end{remark}

\begin{remark}
If we fix an eigensystem corresponding to a finite slope form, then, comparing with \cite {emerton:local-global-conjecture}, we find that our maps\footnote{In the ordinary split case, one must invoke the existence of an overconvergent pre-image under $\theta^{k-1}$ for the evil twin and then also apply (\ref{eq:ocmf-verma-pairing}) in weight $2-k$, just as in \cite[Remark 7.6.3]{emerton:local-global-conjecture}} witness in most cases the full locally analytic Jacquet module of the corresponding representation of $\GL_2(\bbQ_p)$.
\end{remark}

\subsubsection{A simple proof}

The main point in Theorem \ref{maintheorem:ocmf-cup} is the injectivity on the generating highest weight vectors, and it admits a remarkably simple proof: we take the long-exact sequence in continuous group cohomology for the $N_0$-action on the short exact sequence (\ref{eqn:intro-fund-ses}), from which we extract a map
\begin{equation}\label{eq:intro-N0-sequence} H^1(X, \calI)^{N_0} \rightarrow  H^1(N_0, H^{0,\dagger}(X_{\{0\}}, \calI)). \end{equation}
To verify a class is non-trivial, we can apply (\ref{eq:intro-N0-sequence}) then restrict to a neighborhood of either a cusp or an ordinary point, where the action of $N_0$ is explicit. By the construction of (\ref{eq:ocmf-verma-pairing}), we have a function representing the cohomology class of the image of a highest weight vector in $H^1(X, \calI)$ and a precise formula for the action of $N_0$ on it; thus we can compute exactly a representing cocycle for the restriction of its image under (\ref{eq:intro-N0-sequence}) to verify the class is nonzero.

\begin{remark}
The majority of this material was worked out several years ago, but the author was unable to prove this  injectivity at the time. The idea to use (\ref{eq:intro-N0-sequence}) was inspired by Ana Caraiani's talk in the Recent Advances in Modern $p$-adic Geometry (RAMpAGe) seminar on August 6, 2020 on her joint work with Elena Mantovan and James Newton. They use the restriction map on $\calO^+/p^n$-cohomology from the diamond $X/N_0$ to the Igusa variety $X_{\{0\}}^\circ/N_0$ to compare the ordinary part of completed cohomology with ordinary $p$-adic modular forms \`{a} la Hida. Our map in continuous group cohomology is a pedestrian reinterpretation of this restriction map that is well-adapted to our setup.  
\end{remark}

\subsubsection{A Galois corollary}

We say a weight $\kappa$ is defined over $\overline{\bbQ_p}$ if $\Lie \kappa \in \overline{\bbQ_p}$ (equivalently, $\kappa$ is valued in $\overline{\bbQ_p}$). Any such $\kappa$ is valued in a finite extension of $\bbQ_p$, thus it makes sense to discuss overconvergent modular forms of weight $\kappa$ defined over $\overline{\bbQ_p}$ (any one of which is defined over some finite extension of $\bbQ_p$).

 In particular, if we fix an overconvergent modular form $f$ defined over $\overline{\bbQ_p}$, then all of the spaces and maps involved in pairing with $f$ via (\ref{eq:ocmf-verma-pairing}) are defined already over a finite extension of $\bbQ_p$, and we deduce that the image lies in the Hodge-Tate weight zero part of $\tilde{H}^{1}_{c}$. Using this observation we establish the following result, which was conjectured by Gouvea \cite[Conjecture 4]{gouvea:continuity} for weights $k \in \bbZ$ without the hypothesis on the residual representation, and was previously known for finite slope and classical forms. 

\begin{maincorollary} \label{c:hts-weights} If $f$ is an overconvergent eigenform defined over $\overline{\bbQ_p}$ of weight $\kappa$ with $\Lie \kappa \neq 1$, and the attached Galois representation $\rho$ is such that $\overline{\rho}$ is absolutely irreducible, then $\rho_f$ has Hodge-Tate-Sen weights $(0, \Lie \kappa - 1)$.   
\end{maincorollary}

\begin{remark} This result without the restriction $\Lie \kappa \neq 1$ and under the weaker hypothesis that $\rho$ be irreducible is shown in \cite[Theorem 1.0.7]{pan:locally-an}, \textit{along with a converse} (under some minor hypotheses). \end{remark}

\subsubsection{The Hodge-Tate weight zero part of locally algebraic vectors}\label{sss:intro-loc-alg}
\newcommand{\localg}{\mathrm{loc-alg}}
In the following, $K_p \subset \GL_2(\bbZ_p)$ is a compact open subgroup, and $V$ is the $\bbQ_p$-local system on the (open) modular curve $Y_{K_p}$ attached to the Tate module of the universal elliptic curve. By a computation of Faltings \cite{faltings:hodge-tate},
\[ H^1_{\et,c} (Y_{K_p}, (\Sym^k V)^* ) \otimes \bbC_p = S_{k+2}^{\cl, K_p}(1-k) \oplus (M_{k+2}^{\cl, K_p})^* \]
where $M_{k+2}^\cl$ is the space of classical modular forms, the superscript $K_p$ denotes invariants, and the parentheses denote Tate twists. Using the identification $\det V = \bbQ_p(1)$ and applying Emerton's \cite[Theorem 7.4.2]{emerton:local-global-conjecture} computation of locally algebraic vectors in $\tilde{H}^1_{c}$, we conclude that the Hodge-Tate weight zero part of $(\tilde{H}^1_{c})^\localg \otimes \bbC_p$ is identified with
\begin{equation}\label{eq:loc-alg-ht-0} \left(\bigoplus_{k \geq 0} S_{k+2}^\cl \otimes \left( (\Sym^k V)^*\otimes \det V \right)\right)   \oplus \left( \bigoplus_{k \geq 0} (M_{k+2}^{\cl})^\vee \otimes  \Sym^k V  \right). \end{equation}
If $V=\Gamma(\bbP^1, \calO(1))$, then Serre duality identifies 
\[ (\Sym^k V)^*\otimes \det V = H^1(\bbP^1, \calO(-k) \otimes \Omega_{\bbP^1}) \otimes \det V = H^1(\bbP^1, \calO(-k-2)),  \]
where here we have used the \emph{equivariant} identification $O(-2)=\Omega_{\bbP^1} \otimes \det V$. Thus, we may rewrite the first summand in (\ref{eq:loc-alg-ht-0}) as 
\[  \bigoplus_{k \geq 2} S_{k}^\cl \otimes H^1(\bbP^1, \calO(-k)). \]
This is exactly the part recovered by our cup product map as in Theorem \ref{maintheorem:ocmf-cup}. On the other hand, the second half can be rewritten as
\begin{equation}\label{eq:second-half-cup} \bigoplus_{k \geq 0} \colim_{K_p} H^1(X_{K_p}, \omega^{-k} \otimes \calI) \otimes H^0(\bbP^1, \calO(k)). \end{equation}
This can realized similarly as a cup product at infinite level where the roles of $H^1$ and $H^0$ are reversed between the infinite level modular curve and $\bbP^1$. 

\begin{remark}\label{remark:dual-picture}
One can also interpolate the cup products giving (\ref{eq:second-half-cup}) by pairing the local cohomology of $\omega^{\kappa}\otimes \calI$ on the ordinary locus of finite level modular curves with overconvergent sections on $\bbP^1$. Using a method similar to the proof of Theorem \ref{maintheorem:ocmf-cup} we can prove a non-degeneracy for these pairings, but unfortunately only in integral weight\footnote{Our method here hinges on having a section representing the local cohomology class that is defined on (some part of) the opposite ordinary locus -- in particular, to obtain such a representative, the sheaf itself must extend across the supersingular locus}.

Moreover, in integral weight any element in the kernel of the surjection from local cohomology to global cohomology is represented by a modular form on the complement, and after applying the non-trivial Weyl element we obtain an overconvergent modular form inducing the same map via the pairing considered in Theorem \ref{maintheorem:ocmf-cup} -- thus, outside of classical eigensystems, we see nothing new in classical weight, so we have not included the details. 

We note that these local cohomology groups on the modular curve side are essentially those studied by Boxer-Pilloni \cite{boxer-pilloni:higher-hida-modular}. Moreover, they are Serre-dual to overconvergent modular forms, and we expect that this pairing or a variant may fill in the missing geometric description of the space $M_{\mu,1}$ for non-integral weights appearing in the Shimura isomorphism described in \cite[Theorem 5.4.2]{pan:locally-an} (cf. also the paragraph following \cite[Theorem 1.0.1]{pan:locally-an}), completing the analogy with the classical Eichler-Shimura theory described there. 
\end{remark}

\subsubsection{Variant in classical weight}\label{subsec:variants-analyticity}

For classical weights, it is perhaps more  natural to replace $0$ with all of $\bbP^1(\bbQ_p)$ so that we obtain $\GL_2(\bbQ_p)$-equivariant cup products
\[ H^{0,\dagger}(X_{\bbP^1(\bbQ_p)}, \omega^k \otimes \calI)^\sm \otimes H^1_{\bbP^1(\bbQ_p)}(\bbP^1, \calO(-k))  \rightarrow \tilde{H}^1_{c, \bbC_p}. \]
These are related to our previous considerations in the following way: the obvious restriction map realizes $H^{0,\dagger}(X_{\bbP^1(\bbQ_p)}, \omega^k \otimes \calI)^\sm$ as the smooth induction of $S_k^\dagger$. Moreover, there is a natural inclusion 
\[ H^1_{\{0\}}(\bbP^1,  \calO(-k)) \hookrightarrow H^1_{\bbP^1(\bbQ_p)}(\bbP^1,  \calO(-k)),\]
and restricting to $H^1_{\{0\}}(\bbP^1, \calO(-k))$  factors through the map to $S_k^\dagger$ on the left. 

One can furthermore define dual pairings as in Remark \ref{remark:dual-picture} in this setting, and the spaces of overconvergent sections on $\bbP^1$ appearing in these dual pairings are related by an analytic local duality to the local cohomology groups $H^1_{\bbP^1(\bbQ_p)}(\bbP^1, \calO(-k))$, essentially as in work of Morita \cite[Section 5]{morita:anatlyic-discrete-ii}.

\subsection{Organization} In \S\ref{sec:prelim} we setup basic notation, introduce the overconvergent sheaves $\calO(\kappa)$ on $\bbP^1$, and study their local cohomology. In \S\ref{sec:modular} we recall some facts about overconvergent modular forms and establish the fundamental exact sequence (\ref{eqn:intro-fund-ses}). In \S\ref{sec:pairings} we construct the cup product pairings and prove most of Theorem~\ref{maintheorem:ocmf-cup}. Finally, in \S\ref{sec:hodge-tate} we discuss the comparison between completed cohomology and the cohomology of $\calI$ and verify that certain operations in representation theory commute with passage to the Hodge-Tate weight zero part, allowing us to complete the proof of Theorem \ref{maintheorem:ocmf-cup} and prove Corollary \ref{c:hts-weights}.  

\subsection{Acknowledgements}
We thank Ana Caraiani, Matt Emerton, and Kiran Kedlaya for helpful conversations and correspondence. We thank Lue Pan for sharing his preprint \cite{pan:locally-an} and correspondence about the relation with the present work. We thank an anonymous referee for their careful reading and helpful suggestions. This material is based upon work supported by the National Science Foundation under Award No. DMS-1704005.

\section{Constructions on $\bbP^1$}\label{sec:prelim}

\subsection{Conventions}
\subsubsection{Groups and representations}
We write $\GL_2$ for the algebraic group of invertible $2 \times 2$ matrices and $B \subset \GL_2$ for the subgroup of invertible upper-triangular matrices. Let 
\[ N_0 = \begin{bmatrix} 1 & \bbZ_p \\ 0 & 1 \end{bmatrix} \subset B(\bbQ_p). \]
We write $\gl_2$ (resp. $\mathfrak{b}$) for the Lie algebra of $\GL_2$ (resp. $B$) over $\bbQ_p$, and
\[ n_{+} = \begin{bmatrix} 0 & 1 \\ 0 & 0 \end{bmatrix},\; n_{-} = \begin{bmatrix} 0 & 0 \\ 1 & 0 \end{bmatrix} \]
for the raising and lowering operators, respectively, in $\gl_2$. A $(\gl_2, B(\bbQ_p))$-module is a barreled locally-convex Hausdorff $\bbQ_p$-vector space\footnote{For our purposes, we will only need Banach spaces, Fr\'{e}chet spaces, and colimits thereof. We will often be working with representations on vector spaces over a complete extension $E/\bbQ_p$, but, e.g., a representation on an $E$-Banach space is in particular a representation on a $\bbQ_p$-Banach space! In particular, even though $E$ will often be $\bbC_p$, we will never perform any operations that require caution over the fact that $\bbC_p$ is not spherically complete.} $V$ equipped with an action of $\gl_2$ and a locally analytic representation of $B(\bbQ_p)$ such that 
\begin{enumerate}
\item the two induced actions of $\mathfrak{b}$ on $V$ agree, and
\item for any $X \in \gl_2$, $b\in B(\bbQ_p)$, and $v \in V$,  $b \cdot X(v) = (bX b^{-1})(b \cdot v)$.   
\end{enumerate}
We refer the reader to \cite{schneider-teitelbaum:locally-an} for the notion of a locally analytic representation; however, the representations we consider are so explicit that no general theory is really necessary to make sense of the computations.

\subsubsection{Geometry}
In the following, $E$ is a complete extension of $\bbQ_p$ and $\bbP^1$ denotes the projective line over $E$, viewed as an adic space over $\Spa(E, \calO_E)$. We take the dual action on $\bbP^1$, so that $\GL_2$ acts via the standard representation on $\calO(1)$ in the basis $x, y$: 
\[ \textrm{for } \gamma = \begin{bmatrix}a & b \\ c & d\end{bmatrix}, \; \gamma \cdot x = ax+cy \textrm{ and } \gamma \cdot y = bx+dy. \]
Writing $z=x/y$ for the standard local coordinate at $0=[0:1] \in \bbP^1$, we have 
\begin{equation}\label{eq:action-on-z} \gamma \cdot z = \frac{az+c}{bz+d}. \end{equation}
We will denote subspaces of $\bbP^1$ defined using $|z|$ with subscripts, e.g. $\bbP^1_{|z| \leq 1/p}$ for the affinoid ball of radius $1/p$ around $0$. Note that when we write $|z| \leq 1/p^n$ what we really mean is $|z|\leq|p^n|$; that is, $1/p^n$ should be interpreted as lying in $|\bbQ_p^\times|$. 

\subsection{Overconvergent line bundles}
\subsubsection{Reduction of structure group}
The geometric torsor of bases for $\calO(-1)$ on $\bbP^1$ is the projection map
\begin{equation}\label{eq:torsor-projection} \bbA^2 \bs \{(0,0)\} \rightarrow \bbP^1. \end{equation}
We will consider the following reductions of structure group: for $\epsilon=1/p^n$, $n \geq 1$, we write $\bbZ_{p,\epsilon}$ for the affinoid $\epsilon$-neighborhood of $\bbZ_p$ in $\bbA^1$. In other words, $\bbZ_{p,\epsilon}$ is the disjoint union of affinoid disks $|z-k|\leq |p^n|$ of radius $1/p^{n}$ as $k$ varies over a set of representatives for the residue classes $\bbZ_p/p^n$. 

We write $\bbZ_{p,\epsilon}^\times$ for the units in $\bbZ_{p,\epsilon}$, or, equivalently, the affinoid $\epsilon$-neighborhood of $\bbZ_p^\times$. 
The restriction of (\ref{eq:torsor-projection}) to $\bbZ_{p,\epsilon} \times \bbZ_{p,\epsilon}^\times \subset \bbA^2\bs\{0\}$ is a geometric $\bbZ_{p,\epsilon}^{\times}$-torsor
\[ T_{\epsilon} \rightarrow \bbP^1_{|z| \leq \epsilon.} \]
It is equivariant for the action of $\Gamma_0(p^n) = \begin{bmatrix} \bbZ_p^\times & \bbZ_p \\ p^n \bbZ_p & \bbZ_p^\times \end{bmatrix} \subset \GL_2(\bbZ_p).$
The corresponding sheaf of sections $\calT_\epsilon \subset \calO(-1)$ is a sheaf-theoretic reduction of structure group for $\calO(-1)$ from $\bbG_m$ to $\bbZ_{p,\epsilon}^\times$. We fix a canonical trivializing section over $\bbP^1_{z \leq \epsilon}$:
\[ e: [z : 1]  \mapsto (z, 1). \]
These trivializations are compatible along the natural inclusions for $\epsilon' \leq \epsilon$, 
\[ \calT_{\epsilon'} \hookrightarrow \calT_{\epsilon}|_{\bbP^1_{z \leq \epsilon'}}. \]

\subsubsection{Overconvergent sheaves}
Fix $\kappa$ a continuous character of $\bbZ_{p}^\times$ valued in $E$. Any such character extends uniquely to $\bbZ_{p,\epsilon}^\times$ for $\epsilon$ sufficiently small. For such an $\epsilon$, we obtain a line bundle on $\bbP^1_{|z| \leq \epsilon}$ by pushing out via the reciprocal character $\kappa^{-1}$: 
\[ \calO(\kappa) := \calT_{\epsilon} \times_{\bbZ_{p,\epsilon}^\times, \kappa^{-1}} \calO \]
 (independent of the choice of $\epsilon$ up to canonical isomorphism).  The canonical trivialization $e$ of $\calT_{\epsilon}$ gives a canonical trivialization $e_\kappa = (e, 1)$ of $\calO(\kappa)$. When $\kappa$ is the character $z \mapsto z^{k}$, we have a canonical identification $\calO(\kappa)=\calO(k)|_{\bbP^1_{|z| \leq \epsilon}}$, and the section $e_\kappa$ of $\calO(\kappa)$ is identified with $y^k$. In general, 
\begin{equation}\label{eq:can-base-transf} \begin{bmatrix}a & b \\ c & d \end{bmatrix} \cdot e_\kappa = \kappa(bz+d) e_\kappa. \end{equation}

We note that $\calO(\kappa)$ has an obvious action of $\Gamma_0(p^n)$ for $n$ sufficiently large. In fact, the $B(\bbZ_p)$-action extends naturally to an action of $\begin{bmatrix} \bbQ_p^\times & \bbQ_p \\ 0 & \bbZ_p^\times \end{bmatrix}$ on the \emph{germ} of $\calO(\kappa)$ at $0$: if we fix $\epsilon$ sufficiently small to define $\omega^{\kappa}$ via the torsor $T_\epsilon$, then one can check that for any $\gamma$ in this group, there is an $\epsilon_\gamma \leq \epsilon$ such that the action of $\gamma$ on $\bbA^2 \bs \{0\}$ restricts to a map $T_{\epsilon}|_{z \leq \epsilon_{\gamma}} \rightarrow T_{\epsilon}$.  We extend this to an action of $B(\bbQ_p)$ by letting $\diag(p,p)$ act trivially.
\begin{remark} This essentially arbitrary choice for the action of $\diag(p,p)$ will not matter in any of our final results, though it is desirable to make a choice that assigns inverse values for the action of $\diag(p,p)$ on $\calO(\kappa)$ and $\calO(\kappa^{-1})$ so that the canonical identification
$\calO(\kappa)^* = \calO(\kappa^{-1})$ is $B(\bbQ_p)$-equivariant. Indeed, the choice is cancelled out in all of our main results because the representations appearing (cf. e.g., (\ref{eq:first-version-ocmf-pairing})), are always a tensor of one space of sections constructed from $\calO(\kappa)$ and another from $\calO(\kappa)^*=\calO(\kappa^{-1})$ -- thus, no matter what choice we make here, $\diag(p,p)$ will act trivially on the resulting representation.

For $\kappa=z^k$, it would be more natural in some ways to let $\diag(p,p)$ act by $p^k$ because then this would agree via restriction with the standard $\GL_2(\bbQ_p)$-equivariant structure on $\calO(k)$. However, there is no natural way to interpolate this choice for other $\kappa$, and in the end it is perhaps more useful to have uniform formulas by always taking the action to be trivial.  
\end{remark}

It will be useful to have an explicit formula for this action: in a sufficiently small neighborhood of $0$, if we write
\[ \calO(\kappa) = \calO_{\bbP^1} \cdot e_\kappa \]
then from (\ref{eq:action-on-z}), (\ref{eq:can-base-transf}), and the fact that $\diag(p,p)$ acts trivially, we find
\begin{align}  \begin{bmatrix}a & b \\ c & d\end{bmatrix} \cdot (f(z) e_\kappa) & = \left( \begin{bmatrix}a & b \\ c & d\end{bmatrix} \begin{bmatrix}|d| & 0 \\ 0 &|d| \end{bmatrix} \right)\cdot(f(z)e_\kappa) \\ & =  f\left( \frac{az + c}{bz+d} \right) \kappa(b|d| z + d|d|) e_\kappa \label{eq.action-formula-for-sections} \end{align}
whenever the final equation ``makes sense," that is, when $f$ and the locally analytic function $\kappa$ are defined on their inputs -- in particular, by inspection of the formula we find that for any fixed matrix in $B(\bbQ_p)$ and $f$ defined on a fixed ball around 0, the result is defined on some fixed  (potentially smaller) ball around 0, just as we knew already from the above conceptual interpretation of the action!

\subsection{Local cohomology representations}\label{subsub.local-cohomology}

\subsubsection{Recollections on local cohomology} 
For $Z \subset \bbP^1$ a closed set and $\calF$ a sheaf defined on an open neighborhood $U \supset Z$, we consider the local cohomology\footnote{or relative cohomology, or cohomology with supports, depending on your tastes!} group $H^1_{Z}(U, \calF)$ (for the analytic topology on $\bbP^1$). It fits into a functorial exact sequence
\begin{equation}\label{eq.local-cohomology-exact-sequence} 0 \rightarrow H^0(U, \calF)/H^0(U\bs Z, \calF) \rightarrow H^1_{Z}(U, \calF) \rightarrow H^1(U\bs Z, \calF). \end{equation}
For $Z \subset U' \subset U$ there is a natural restriction map 
\[ H^1_{Z}(U, \calF) \rightarrow H^1_{Z}(U', \calF) \]
and by the excision property it is an isomorphism (this will be obvious by direct computation in all cases we consider). Moreover, (\ref{eq.local-cohomology-exact-sequence}) is compatible with the restriction maps on all terms, and restriction satisfies the natural compatibilities with respect to $Z \subset U'' \subset U' \subset U.$ In particular, we can unambigulously write
\[ H^1_{Z}(\bbP^1, \calF) := H^1_Z(U, \calF) \]
for any choice of $U$ as above, and this abuse of notation makes sense even though $\calF$ may not be defined on all of $\bbP^1$ (in particular, we will apply this to the sheaves $\calO(\kappa)$ of the previous section). 

We note further that for $Z' \subset Z$, there is a corestriction map 
\[ \mathrm{cores}: H^1_{Z'}(\bbP^1, \calF) \rightarrow H^1_{Z}(\bbP^1, \calF) \]
and these corestrictions satisfy the obvious compatibility for $Z'' \subset Z' \subset Z.$ Moreover, for any choice of $U$ the exact sequence (\ref{eq.local-cohomology-exact-sequence}) is functorial for $Z' \subset Z \subset U$ via corestriction in the middle and the natural restriction maps on the other terms that appear. 

\subsubsection{Representations on local cohomology}
We now consider the local cohomology groups of the sheaves $\calO(\kappa)$ along the closed subsets $S=\{0\}$ and $S=|z| < \epsilon$ for $\epsilon=1/p^n$ sufficiently small. Here $|z|<\epsilon$ must be read in the language of adic spaces -- it is then the complement of the open affinoid $|z| \geq \epsilon$, thus closed. Its \emph{interior} is the adification of the rigid analytic ball defined by the same inequality, and the boundary is a single rank two point where $|z|$ is infinitesimally smaller than $|p^n|$. 

We state the main properties of these groups as a lemma; in the proof we will use (\ref{eq.local-cohomology-exact-sequence}) to make the groups themselves completely explicit.
\begin{lemma}\label{lemma.local-cohomology-properties}\hfill
\begin{enumerate} 
\item For $\epsilon=1/p^n$ sufficiently small, $H^1_{|z|<\epsilon}(\bbP^1, \calO(\kappa))$ is canonically an orthonormalizeable $E$-Banach space with a locally analytic action of $\Gamma_0(p^{n+1})$. 
\item For $\epsilon'=1/p^m<\epsilon=1/p^n$, the corestriction map
\[  H^1_{|z| < \epsilon'}(\bbP^1, \calO(\kappa)) \rightarrow H^1_{|z| < \epsilon}(\bbP^1, \calO(\kappa)) \]
is injective, completely continuous (i.e. a compact operator) and $\Gamma_0(p^{m+1})$-equivariant. In particular, it is a map of $(\gl_2, B(\bbZ_p))$-modules. 
\item The corestriction maps 
\[ H^1_{\{0\}}(\bbP^1, \calO(\kappa)) \rightarrow H^1_{|z|<\epsilon}(\bbP^1, \calO(\kappa)) \]
are injective and induce an isomorphism 
\begin{equation} \label{eq.local-cohomology-limit} H^1_{\{0\}}(\bbP^1, \calO(\kappa)) = \lim_{\epsilon} H^1_{|z|<\epsilon}(\bbP^1, \calO(\kappa)). \end{equation}
In particular, $H^1_{\{0\}}(\bbP^1, \calO(\kappa))$ is a nuclear Fr\'{e}chet space with the structure of a $(\gl_2, B(\bbZ_p))$-module. Moreover, the action of $B(\bbQ_p)$ on the germ of $\calO(\kappa)$ at $0$ extends this to the structure of a $(\gl_2, B(\bbQ_p))$ module. 
\end{enumerate}
\end{lemma}
\begin{remark}
The dual of the nuclear Fr\'{e}chet space $H^1_{\{0\}}(\bbP^1, \calO(\kappa))$ is canonically identified with the space of compact type $H^{0,\dagger}(\{0\}, \calO(\kappa^{-1}))$ via the residue pairing (cup to $H^1_{\{0\}}(\bbP^1, \calO)$ then pass to $H^1(\bbP^1, \calO)=E$).  
\end{remark}

\begin{proof}

For $\epsilon_0$ sufficiently small and $\epsilon \leq \epsilon_0$, (\ref{eq.local-cohomology-exact-sequence}) with $U=\bbP^1_{ |z| \leq \epsilon_0}$ gives
\begin{equation}\label{eq.ball-local-computation}
H^1_{|z| < \epsilon}(\bbP^1, \calO(\kappa))=	H^0(\bbP^1_{\epsilon \leq |z|\leq \epsilon_0}, \calO(\kappa)) / H^0(\bbP^1_{|z| \leq \epsilon_0}, \calO(\kappa)). \end{equation}
Indeed, $\calO(\kappa)$ is a line bundle, so its $H^1$ vanishes on the affinoid $\bbP^1_{\epsilon \leq |z| \leq \epsilon_0}.$ 

Using that $\calO(\kappa)=\calO_{\bbP^1} \cdot e_\kappa$, we see that
\[ H^0(\bbP^1_{\epsilon \leq |z|\leq \epsilon_0}, \calO(\kappa)) = \left\{ \left(\sum_{k \in \bbZ} a_k z^k \right) e_\kappa \;  |\, \lim_{k \rightarrow \infty,\, k\geq 0} |a_k| \epsilon_0^k  = 0 \textrm{ and } \lim_{k \rightarrow -\infty,\, k< 0} |a_k| \epsilon^k =0 \right\}.\]
This is an orthonormalizeable Banach space with norm given by the sup over the terms appearing in the two limits. 
The subspace $H^0(\bbP^1_{|z| \leq \epsilon_0}, \calO(\kappa))$ is closed; indeed, it is given by those power series with $a_k=0$ for $k<0$. The quotient $H^1_{|z| < \epsilon}(\bbP^1, \calO(\kappa))$ is thus also a Banach space, and the induced norm can be computed on any representative 
\[ \left(\sum_{k \in \bbZ} a_k z^k \right) e_\kappa \]
as $\sup_{k < 0} |a_k| \epsilon^k.$ In particular, this norm is independent of the choice of $\epsilon_0$, and we also immediately obtain that the corestriction maps for $\epsilon' < \epsilon$ are compact. 

The $\Gamma_0(p^{n+1})$ action for $\epsilon=1/p^{n}$ exists because $\Gamma_0(p^{n+1})$ preserves the ball $|z| < \epsilon$, and it is an immediate computation from (\ref{eq.action-formula-for-sections}) that it is locally analytic (using the  local analyticity of $\kappa$). The equivariance for corestriction is immediate, so we have now established points \emph{1.} and \emph{2.} in the statement of the lemma. 

To prove \emph{3.}, we again apply (\ref{eq.local-cohomology-exact-sequence}) with $U=\bbP^1_{ |z| \leq \epsilon_0}$ to see
\[  H^1_{\{0\}}(\bbP^1, \calO(\kappa))=	H^0(\bbP^1_{0 < |z|\leq \epsilon_0}, \calO(\kappa)) / H^0(\bbP^1_{|z| \leq \epsilon_0}, \calO(\kappa)). \]
Indeed, $\calO(\kappa)$ is a line bundle, so its $H^1$ vanishes on the quasi-Stein $\bbP^1_{0 < |z| \leq \epsilon_0}.$ 

The limit formula (\ref{eq.local-cohomology-limit}) is then immediate by comparing with (\ref{eq.ball-local-computation}) and using the trivial identity 
\[ H^0(\bbP^1_{0 < |z|\leq \epsilon_0}, \calO(\kappa)) = \lim_{\epsilon} H^0(\bbP^1_{\epsilon \leq |z|\leq \epsilon_0}, \calO(\kappa)). \]
It remains only to check that the $B(\bbQ_p)$ action induced by the action on the germ satisfies the correct compatibility with the $\gl_2$-action under the adjoint action, and this can be verified, e.g., with the explicit formula (\ref{eq.action-formula-for-sections}). 
\end{proof}

\subsubsection{The Verma module}\label{subsub.verma-module}
The computations in the proof of Lemma \ref{lemma.local-cohomology-properties} show that $H^1_{0}(\bbP^1, \calO(\kappa))$ can be unambiguously identified with the the set of tails
\[ \left\{\left( \sum_{k<0} a_k z^k \right) e_\kappa \; | \, \forall t>0, \lim_{k\rightarrow \infty} |a_k|t^k = 0  \right\}. \]
In this presentation, the Fr\'{e}chet topology is given by the norms for $t>0$
\[ || \left( \sum_{k<0} a_k z^k \right) e_\kappa ||_t = \sup |a_k|t^k. \] 
The actions of $\gl_2$ and $B(\bbQ_p)$ are then obtained by applying (\ref{eq.action-formula-for-sections}) to compute naively and then truncating the resulting Laurent series back to its tail. In particular, short computations show
\begin{enumerate}
\item $n_{-}$ acts as $\frac{d}{dz}$,
\item $z^{-1} e_\kappa$ is a highest weight vector of weight $(-1, \Lie \kappa + 1)$, and
\item the action of $B(\bbQ_p)$ preserves the set of tails truncated at any finite power. 
\end{enumerate}
It follows that the subspace $V_{\kappa}$ of \emph{meromorphic} tails is a $(\gl_2, B(\bbQ_p))$-module (topologized with the colimit topology over the finite truncations) whose underlying $\gl_2$-representation is the Verma module of highest weight $(-1, \Lie \kappa + 1)$. In particular, when  $\Lie \kappa=k \in \bbZ_{\leq -2}$, it admits an algebraic quotient and the kernel $V_\kappa'$ has highest weight vector $z^{k} e_\kappa$ of weight $(k, 0)$.  For $\kappa=z^k$, $k \leq -2$, $V_{\kappa}$ is the \emph{algebraic} local cohomology group, the quotient is $H^1(\bbP^1,\calO(k))$, and the kernel $V_\kappa'$ is $H^0_{\mathrm{alg}}(\bbP^1 \bs\{0\}, \calO(k))$  with $z^k e_\kappa= x^k$. 

\section{Constructions on modular curves}\label{sec:modular}

Recall from \S\ref{ss:summary} that $X$ is the perfectoid infinite level compactified modular curve of prime-to-$p$ level $K^p$ of \cite{scholze:torsion}. It admits a Hodge-Tate period map
\[ \pi_{\HT}: X \rightarrow \bbP^1 \]
such that over the open perfectoid modular curve $Y$, the pullback of 
\[ 0 \rightarrow \calO(-1)\otimes \det \rightarrow \Gamma( \calO(1) ) \otimes \calO \rightarrow \calO(1) \rightarrow 0 \]
is $\GL_2(\bbQ_p)$-equivariantly identified with the Hodge-Tate sequence for the Tate module of the universal elliptic curve, 
\begin{equation}\label{eq:hodge-tate-ext} 0 \rightarrow \omega^{-1}(1) \rightarrow \calO^2 \rightarrow \omega \rightarrow 0. \end{equation}
Here we use the canonical trivialization of the Tate module at infinite level, so that the $GL_2(\bbQ_p)$-action on the middle term is via the standard representation.

\begin{remark}
The notation $\calO(k)$ on $\bbP^1$ refers to the standard line bundles, whereas $\omega^{-1}(i)$ on $X$ denotes a Tate twist of the sheaf $\omega^{-1}$ on $X$. This should not cause any confusion in what follows, as we will never use Tate twists on $\bbP^1$. Instead, Tate twists can be replaced on $\bbP^1$ by twisting the $\GL_2(\bbQ_p)$-action via the determinant because, on $X$, $\bbQ_p(1)$ has a trivialization transforming via the determinant under $\GL_2(\bbQ_p)$ -- indeed, $\bbQ_p(1)$ is the determinant of the Tate module of the universal elliptic curve.
\end{remark}

\subsection{Overconvergent modular forms}
As in the introduction, we define 
\[ \omega^{\kappa} = \pi_{\HT}^{-1}\calO(\kappa) \textrm{ and } S_\kappa^\dagger = H^{0,\dagger}(X_{\{0\}}, \omega^{\kappa} \otimes \calI)^\sm. \]
The action of $B(\bbQ_p)$ on the germ of $\calO(\kappa)$ at $0$ equips $S_\kappa^\dagger$ with a natural $(\gl_2, B(\bbQ_p))$-module structure, with $\gl_2$ acting trivially. 
 
We briefly recall (cf. \cite{chojecki-hansen-johannson:ocmf, howe:thesis, birkbeck-heuer-williams:hilbert}) why the $B(\bbZ_p)$-invariants in $S_\kappa^\dagger$ agree with the classical definition of overconvergent modular forms of weight $\kappa$: to show the line bundles $\omega^\kappa$ descend to finite level, one first observes that the torsor of bases $\calT_{\epsilon}|_{X_{|z| \leq \epsilon}}$ has a $K_p$-invariant section for $K_p \subset \Gamma_0(p^n)$ sufficiently small, which follows from the density of $\calO_X(X_{|z| \leq \epsilon})^\sm$ in $\calO(X_{|z| \leq \epsilon})$ established in \cite[Theorem 3.1.2-(iii)]{scholze:torsion}. This allows one to descend $\calT_\epsilon$ to an \'{e}tale\footnote{it is only \'{e}tale because the $K_p$ with an invariant section may not equal $\Gamma_0(p^n)$} torsor away from a neighborhood of the boundary in the image of $X_{|z| \leq \epsilon}$ in the modular curve of finite level $\Gamma_0(p^n)K^p$. One argues by hand to extend to the cusps, as over the canonical component of the ordinary locus $e_\kappa$ agrees with the standard Katz trivialization. The resulting torsor at finite level can now be pushed out to give a descent of $\omega^\kappa$, and the overconvergent sections along the canonical ordinary locus will be exactly the $B(\bbZ_p)$-invariants in $S_\kappa^\dagger$; the choice of the specific $n$ in $\Gamma_0(p^n)$ is irrelevant here as passing to a sufficiently small $\epsilon$ allows one to move between $n$ (classically via overconvergence of the canonical subgroup, or from this perspective just using that the orbit of $|z| \leq \epsilon$ under any $\Gamma_0(p^n)$ is a disjoint union of translates of the ball).

At this point, one can, e.g., compare with the construction of Pilloni \cite{pilloni:ocmf}, which carries out the same idea at finite level to reduce the structure group mod $p^n$ for the Hodge-Tate integral structure. 

\subsection{The fundamental Mayer-Vietoris sequence}

Recall that $\calI$ denotes the ideal sheaf of the boundary (cusps) $\partial X$ of $X$. To obtain the quasi-Stein version of the Mayer-Vietoris exact sequence (\ref{eqn:intro-fund-ses}) we first establish a vanishing lemma:

\begin{lemma}\label{lemma:vanishing} For any $i \geq 1$ and any $\delta=1/p^n > 0$, 
\[ H^i(X_{|z| > 0}, \calI) = 0 \textrm{ and } H^i(X_{\delta \geq |z| > 0}, \calI)=0  \]
\end{lemma}
\begin{proof}
Arguing as in \cite[Theorem 3.1.2-(iii)]{scholze:torsion}, we find that $X_{|z| \geq \epsilon}$ is affinoid perfectoid for any $\epsilon=1/p^N>0$. Indeed: there is some neighborhood $U$ of $\infty$ such that $X_U$ is affinoid perfectoid, and we can spread out to obtain an affinoid perfectoid containing $X_{|z| \geq \epsilon}$ using the action of $\diag(p,1)^{\bbZ}$; then $X_{|z| \geq \epsilon}$ is affinoid perfectoid as a rational subdomain. Similarly, we find $X_{\delta \geq |z| \geq \epsilon}$ is affinoid perfectoid for $\delta \geq \epsilon$ as it is again a rational subdomain. 

With this established, the arguments for the two statements are essentially the same, so we just show that $H^i(X_{|z|>0}, \calI)=0$. Write $U_n= X_{|z| \geq 1/p^n}$, an affinoid perfectoid. Because $\partial X$ is strongly Zariski closed\footnote{This is shown in \cite{scholze:torsion}, but can also be verified by explicit computation for the modular curve, or as a general consequence of \cite[Remark 7.5]{bhatt-scholze:prismatic}, which shows Zariski closed always implies strongly Zariski closed.}, we have an exact sequence
\[ 0 \rightarrow \calI \rightarrow \calO_X \rightarrow \calO_{\partial X} \rightarrow 0 \]
which is moreover exact after evaluation on an affinoid perfectoid. Because $\calO_X$ and $\calO_{\partial X}$ both have vanishing higher cohomology on affinoid perfectoids (the intersection of an affinoid perfectoid in $X$ with $\partial X$ is affinoid perfectoid in the perfectoid space $\partial X$), the long exact sequence of cohomology gives 
\[ H^i(U_n, \calI) = 0 \textrm{ for } i \geq 1. \]

We now carry out the standard argument to bootstrap up to vanishing on the quasi-Stein space $X_{|z| >0} = \bigcup_{i=1}^{\infty} U_i$ (cf. e.g. \cite[Theorem 2.6.5]{kedlaya-liu:relativeII}\footnote{Unfortunately, we cannot just invoke this result, as the ideal sheaf is not pseudocoherent.}). Because the higher cohomology vanishes on $U_n$, we can compute $H^i(X_{|z|>0}, \calI)$ using the \v{C}ech complex for the cover $\{ U_i \}_{i \geq 1}$. This \v{C}ech complex is the limit of the \v{C}ech complexes for $\{ U_i \}_{1 \leq i \leq n}.$ Moreover, the transition maps from $n+1$ to $n$ are surjective in each degree, so the derived limit of this sequence of complexes is the limit (cf. \cite[Tag 091D]{stacks-project}).
The \v{C}ech complex for $\{U_i\}_{1 \leq i \leq n}$ has vanishing higher cohomology (because it computes the cohomology of $\calI$ on $U_n$), thus, applying \cite[Tag 0CQE]{stacks-project}, we conclude that $H^i(X_{|z|>0}, \calI) = 0$ for $i \geq 2$ and 
\[ H^1(X_{|z|>0}, \calI)= R^1 \lim H^0(U_n, \calI). \]

It remains to show this $R^1\lim$ is zero. We write $I_n= H^0(U_n, \calI)$, $I^+_n = H^0(U_n, \calI^+)$,  $A_n=H^0(U_n, \calO)$, and $A_n^+=H^0(U_n, \calO^+)$.  Then $A_n$ is a $\bbC_p$-Banach space with unit ball $A_n^+$, and $I_n$ is a closed subspace, thus a Banach space with unit ball $I_n^+$ (which is equal to $I_n \cap A_n^+$). Because $U_n$ is a Weierstrass subdomain of $U_{n+1}$, the image of $A_{n+1}$ in $A_n$ is dense. It follows from \cite[Lemma 2.2.9-(ii)]{scholze:torsion} that the map $I_{n+1} \rightarrow I_n$ also has dense image. We write  $||\cdot||_n$ for the Banach norm on $I_n$ with unit ball $I_n^+$. We note that for $f  \in I_{n+1} \subset I_n$, $||f||_{n} \leq ||f||_{n+1}$. 

As usual, the first derived limit is computed as the cokernel of the map 
\[ \psi: \prod_{n \geq 1} I_n  \rightarrow \prod_{n \geq 1} I_n, \; (f_n)_n \mapsto (f_n - f_{n+1})_n,\]
so we must show that $\psi$ is surjective. Thus, suppose $(y_n)_n \in \prod_{n \geq 1} I_n$. By density of $I_{n+1}$ in $I_n$, we can inductively choose $x_1 =0$, then $x_2$ such that $|| y_1 - x_2 ||_1 \leq 1/p$, then $x_3$ such that $||y_2 + x_2 - x_3 ||_2 \leq 1/p^2$, etc., so that 
\[ \psi( (x_n)_n ) - (y_n)_n \in \prod_{n\geq 1} B_{1/p^n}(I_n) \]
where $B_{1/p^n}(I_n)$ denotes the ball $|| f ||_n \leq 1/p^n$ in $I_n$. Thus it suffices to show this product is in the image of $\psi$. But, given 
\[ (y_n) \in \prod_{n\geq 1} B_{1/p^n}(I_n) \]
one can construct an explicit inverse $(x_n)$ by setting
\[ x_n = y_n + y_{n+1} + y_{n+2} + \ldots \]
which clearly converges since for any $m \geq n$, $||y_m||_n \leq ||y_m||_m\leq 1/p^m$.  

\end{proof}

\begin{remark} The proof of Lemma \ref{lemma:vanishing} applies to the ideal sheaf of any Zariski closed subset of a perfectoid space which can be written as an increasing union of affinoid perfectoids $U_n$ with $U_{n}$ a Weierstrass localization of $U_{n+1}$. 
\end{remark}

Applying Lemma \ref{lemma:vanishing}, we find that for any $\epsilon=1/p^n > 0$, we can compute $H^1(X, \calI)$ using the \v{C}ech sequence for the cover by $X_{|z| \leq \epsilon}$ and $X_{|z| > 0}$. This is simply a Mayer-Vietoris sequence, and, taking the colimit over $\epsilon \rightarrow 0$ gives
\begin{equation}\label{eq:fund-ex-seq} 0 \rightarrow \substack{ H^{0,\dagger}(X_{\{0\}}, \calI) \\ \oplus H^0(X_{|z| >0 }, \calI)} \rightarrow \colim_{\epsilon>0} H^0(X_{0 < |z| \leq \epsilon}, \calI) \rightarrow H^1(X, \calI) \rightarrow 0. \end{equation}

\section{Pairings}\label{sec:pairings}

\subsection{Global cup products}
For any $\kappa$, there is an obvious pairing
\[ \langle, \rangle: \left(\omega^{\kappa} \otimes \calI\right) \otimes \omega^{\kappa^{-1}} \rightarrow \calI. \]
When $\kappa = z^k$, then combining this with the identification $\omega^k = \pi_{\HT}^{-1}(\calO(k))$ and pullback of sections, we obtain global $\GL_2(\bbQ_p)$ and prime-to-$p$ Hecke-equivariant cup product maps
\begin{equation}\label{eqn:global-cup} H^0(X, \omega^k \otimes \calI)^\sm \otimes H^1(\bbP^1, \calO(-k)) \rightarrow H^1(X, \calI)  \end{equation}
and
\begin{equation} H^1(X, \omega^k \otimes \calI )^\sm \otimes H^0(\bbP^1, \calO(-k)) \rightarrow H^1(X, \calI). \label{eqn:other-global-cup} \end{equation}

The $\GL_2(\bbQ_p)$-representations on the domains of (\ref{eqn:global-cup}) and (\ref{eqn:other-global-cup}) are locally algebraic by construction, thus both maps factor through the locally algebraic vectors for the $\GL_2(\bbQ_p)$-action on $H^1(X, \calI).$ 

\subsection{Local cup product and compatibility}\label{ss:local-cup-compat}
We define a pairing
\begin{equation}\label{eq:ocmf-cup} S_\kappa^\dagger \otimes H^1_{\{0\}}(\bbP^1, \calO(\kappa^{-1})) \rightarrow H^1(X, \calI) \end{equation}
by sending $f \otimes c$ to $[ \langle f, \pi_\HT^*g \rangle ]$, where $g$ is any representative of the class $c$ in  
\begin{equation}\label{eq.local-cohom-surjection} \colim_{\epsilon>0} H^0(\bbP^1_{0 < |z| \leq \epsilon}, \calO(\kappa^{-1})) \twoheadrightarrow H^1_{\{0\}}(\bbP^1, \calO(\kappa^{-1})), \end{equation}
and the square brackets indicate application of the map
\[ \colim_{\epsilon>0} H^0(X_{0 < |z| \leq \epsilon}, \calI) \twoheadrightarrow H^1(X, \calI) \]
appearing in (\ref{eq:fund-ex-seq}).

This is well-defined because for 
 \[ h \in H^{0,\dagger}(\{0\}, \calO(\kappa^{-1})) = \colim_{\epsilon>0} H^0(\bbP^1_{|z| \leq \epsilon}, \calO(\kappa^{-1})),\]
the kernel of (\ref{eq.local-cohom-surjection}), the section $\langle f, \pi_{\HT}^* h \rangle$ extends to an element of $H^{0,\dagger}(X_{\{0\}}, \calI)$ and thus maps to zero in $H^1(X, \calI)$ (cf. (\ref{eq:fund-ex-seq})). 

Recall that we defined in \ref{subsub.verma-module} a Verma module $V_{\kappa^{-1}} \subset H^1_{\{0\}}(\bbP^1, \calO(\kappa^{-1}))$, and, when $\kappa=k \geq 2$, a submodule $V_{-k}' \subset V_{-k}$.

\begin{lemma}\label{lemma:cup-classical-zero}\hfill
\begin{enumerate}
\item For $k \in \bbZ_{\geq 2}$, the restriction of (\ref{eq:ocmf-cup}) to $ S_k^\cl \otimes V_{-k}'$ is zero. 
\item For $k=1$, the restriction of (\ref{eq:ocmf-cup}) to $ S_k^\cl \otimes V_{-1} $
is zero. 
\end{enumerate}
\end{lemma}
\begin{proof}
In both cases the highest weight vector is $z^{-k} y^{-k}=x^{-k}$. Thus, for $s \in S_k^\cl$, 
\[ \langle s, z^{-k}y^{-k} \rangle= \langle s, x^{-k} \rangle \] 
extends to a section on $H^0(X_{|z|>0}, \calI)$ (because $s$ is a global section and $x^{-k}$ has a pole only at $0$), and thus $[\langle s, x^{-k} \rangle]$ vanishes. 
\end{proof}

In particular, the lemma implies that for $k\geq 2$ we obtain an induced pairing
\[ S_k^\cl \otimes H^1(\bbP^1, \calO(-k)) \rightarrow H^1(X, \calI). \]
It is immediate from the explicit description of cup products in terms of \v{C}ech classes that this agrees with the global pairing (\ref{eqn:global-cup}). 

\subsection{Local analyticity, continuity, and equivariance} 
We now show
\begin{lemma}\label{lemma:loc-an-cont-eq} The map (\ref{eq:ocmf-cup}) is continuous and factors through a map of $(\gl_2, B(\bbQ_p))$-modules
\[ S_\kappa^\dagger \otimes H^1_{\{0\}}(\bbP^1, \calO(\kappa^{-1})) \rightarrow H^1(X, \calI)^\locan. \]
\end{lemma}

Here the domain $S_\kappa^\dagger \otimes H^1_{\{0\}}(\bbP^1, \calO(\kappa^{-1}))$ is topologized as the colimit of 
\[ W \otimes H^1_{\{0\}}(\bbP^1, \calO(\kappa^{-1})) \textrm{ for } W \subset S_\kappa^\dagger, \dim W < \infty,   \]
and the codomain $H^1(X, \calI)$ is considered as a Banach space with unit ball the image of $H^1(X, \calI^+)$. The map (\ref{eq:ocmf-cup}) is $B(\bbQ_p)$-equivariant by construction, so Lemma \ref{lemma:loc-an-cont-eq} is reduced to showing that for any fixed $f \in S_{\kappa}^\dagger$, the map induced by pairing with $f$ is continuous as a map $H^1_{\{0\}}(\bbP^1, \calO(\kappa^{-1})) \rightarrow H^1(X, \calI)$ and factors through a $\gl_2$-equivariant map
\[ H^1_{\{0\}}(\bbP^1, \calO(\kappa^{-1})) \rightarrow H^1(X, \calI)^\locan. \]

We will show this by extending our pairing for any \emph{fixed} radius of overconvergence $\epsilon$ to a pairing with the local cohomology of $|z| < \epsilon$ in $\bbP^1$; these same groups were also used to define the $\gl_2$-action on $H^1_{\{0\}}(\bbP^1, \calO(\kappa^{-1}))$ in \ref{subsub.local-cohomology}. The key point is that the $\gl_2$-action on these larger groups is induced by a locally analytic $\Gamma_0(p^n)$-action and the extended pairings will be $\Gamma_0(p^n)$-equivariant by construction; thus we will obtain the desired statement as soon as we also have continuity. We now carry out this strategy:

If we fix $\epsilon=1/p^n > 0$, then we can also compute $H^1(X, \calI)$ using the Mayer-Vietoris sequence for the covering by $X_{|z| \leq \epsilon}$ and $X_{|z| \geq \epsilon}$. The Mayer-Vietoris sequence for the covering by $X_{|z| \leq \epsilon}$ and $X_{|z| > 0}$ maps to it naturally via restriction, and thus we obtain a commutative diagram 
\[ \xymatrix{ & H^0(X_{|z| = \epsilon}, \calI) \ar@{>>}[dr] & \\ H^0(X_{0<|z| \leq \epsilon}, \calI)\ar@{>>}[rr]\ar[ur]\ar[dr] & & H^1(X, \calI) \\ & \colim_{\epsilon'>0} H^0(X_{0 < |z| \leq \epsilon'}, \calI)\ar@{>>}[ur]_{(\ref{eq.local-cohom-surjection})} & }\]
Imitating the steps of the previous section for $\epsilon$ sufficiently small and using the top right arrow in place of (\ref{eq.local-cohom-surjection}), we obtain a pairing
\begin{equation}\label{eq.overconvergent-radius-pairing} H^0(X_{|z|\leq \epsilon}, \omega^{\kappa} \otimes \calI)^\sm \otimes H^1_{|z|<\epsilon}(\bbP^1, \calO(\kappa^{-1})) \rightarrow H^1(X, \calI). \end{equation}
By Lemma \ref{lemma.local-cohomology-properties}, $H^1_{|z|<\epsilon}(\bbP^1, \calO(\kappa^{-1}))$ has a locally analytic action of $\Gamma_0(p^{n+1})$ that induces the $\gl_2$-action on $H^1_{\{0\}}(\bbP^1, \calO(\kappa^{-1}))$. By construction, (\ref{eq.overconvergent-radius-pairing}) is $\Gamma_0(p^{n+1})$-equivariant. 

Also by construction, the diagram
\[ \xymatrix{ H^0(X_{|z|\leq \epsilon}, \omega^{\kappa} \otimes \calI)^\sm \otimes H^1_{|z|<\epsilon}(\bbP^1, \calO(\kappa^{-1})) \ar[dr]^{(\ref{eq.overconvergent-radius-pairing})} & \\ 
 H^0(X_{|z|\leq \epsilon}, \omega^{\kappa} \otimes \calI)^\sm \ar[u]^{\Id \times \mathrm{cores}} \ar[d]_{\mathrm{res} \times \Id}\ar[r] \otimes H^1_{\{0\}}(\bbP^1, \calO(\kappa^{-1})) &   H^1(X, \calI) \\
S_\kappa^\dagger \otimes H^1_{\{0\}}(\bbP^1, \calO(\kappa^{-1})) \ar[ur]_{(\ref{eq:ocmf-cup})} &} \]
comparing the two pairings commutes. Thus, at the price of shrinking our space of overconvergent modular forms by enforcing a radius of convergence, we enlarge the local cohomology group that we are allowed to pair with. 

Now, any $f \in S_\kappa^\dagger$ extends to a smooth section over some $X_{|z|\leq \epsilon}$, $\epsilon=1/p^n$, so the pairing with $f$ extends to $H^1_{|z|<\epsilon}(\bbP^1, \calO(\kappa^{-1}))$. We may choose a compact open subgroup $K_p \subset \Gamma_0(p^{n+1})$ fixing $f$, and the resulting map 
\begin{equation}\label{eq.pairing-with-single-element-extended} H^1_{|z|<\epsilon}(\bbP^1, \calO(\kappa^{-1})) \rightarrow H^1(X, \calI) \end{equation}
is then $K_p$-equivariant. In particular, since the action of $K_p$ on the left is locally analytic, (\ref{eq.pairing-with-single-element-extended}) will factor as a $\gl_2$-equivariant map through $H^1(X, \calI)^\locan$ as soon as we know that it is a bounded (equivalently, continuous) map of Banach spaces -- the point here is that the convergent power series describing the orbit maps on the left will then remain convergent after applying the map!

This boundedness is straightforward after unraveling the definitions: we can write $f = a \cdot e_\kappa$ for $a \in H^0(X_{|z| \leq \epsilon}, \calI)$. Because $X_{|z| \leq \epsilon}$ is affinoid, there exists $N \geq 0$ such that 
\[ b:=p^N a \in H^0(X_{|z| \leq \epsilon}, \calI^+) = H^0(X_{|z|\leq\epsilon}, \calI)\cap H^0(X_{|z|\leq\epsilon}, \calO^+).\]
On the other hand, it is immediate from the description of the Banach norm in \ref{subsub.local-cohomology} that the elements in the unit ball of the Banach space $H^1_{|z| < \epsilon}(X, \calO(\kappa^{-1}))$ are all represented by elements $c \cdot e_\kappa$ for $c \in H^0(\bbP^1_{|z| = \epsilon}, \calO^+)$. Pairing gives
\[ [\langle f, c \cdot e_\kappa^{-1} \rangle ] =  [\langle a \cdot e_\kappa, c \cdot e_{\kappa^{-1}} \rangle ] = \frac{1}{p^n} [ bc ]. \]
Now, $bc \in H^0(X_{|z|=\epsilon}, \calI^+)$, thus $[bc]$ is in the image of $H^1(X, \calI^+)$, the unit ball in $H^1(X, \calI)$. We conclude that the map (\ref{eq.pairing-with-single-element-extended}) is bounded. 

This completes the proof of Lemma \ref{lemma:loc-an-cont-eq}. In particular, we find the restriction of (\ref{eq:ocmf-cup}) to $S_\kappa^\dagger \otimes V_{\kappa^{-1}}$ is a map of $(\gl_2, B(\bbQ_p))$-modules. 

\subsection{Proof of Theorem \ref{maintheorem:ocmf-cup}}

We now prove Theorem \ref{maintheorem:ocmf-cup}, assuming the identification $H^1(X, \calI)=\tilde{H}^1_{c,\bbC_p}$. This identification will be explained in Lemma \ref{lemma:comparison} (there is a bit to say here because only the torsion comparison is given in \cite{scholze:torsion}). 

We begin by verifying that the map is injective on the subspace of generating highest weight vectors, $S_{\kappa}^\dagger \otimes z^{-1} e_{\kappa}^{-1}.$ It suffices to verify it on the $N_0$-invariants,
\[ (S_{\kappa}^\dagger)^{N_0} \otimes z^{-1} e_{\kappa}^{-1},\]
since anything in $S_{\kappa}^\dagger$ can be moved into $(S_{\kappa}^\dagger)^{N_0}$ using the action of $\diag(p,1)^\bbZ$. 

\begin{lemma} For $\Lie \kappa \neq 1$, the restriction of (\ref{eq:ocmf-cup}) to $(S_{\kappa}^\dagger)^{N_0} \otimes z^{-1} e_{\kappa^{-1}}$ is injective .
\end{lemma}
\begin{proof} 
Taking continuous $N_0$-group cohomology in (\ref{eq:fund-ex-seq}), we obtain as part of the boundary a map
\[ \delta: H^1(X, \calI)^{N_0} \rightarrow H^1(N_0, H^{0, \dagger}(X_{\{0\}}, \calI)). \]
We describe it explicitly: For any class $[g] \in H^1(X, \calI)^{N_0}$ and $\gamma \in N_0$, we can express $\gamma\cdot g - g$ uniquely as $a_\gamma - b_\gamma$ for $a_\gamma \in H^{0,\dagger}(X_{\{0\}}, \calI)$ and $b_\gamma \in H^0(X_{|z| >0 }, \calI)$. Then $\delta([g])$ is the class represented by the cocycle $\gamma \mapsto a_\gamma$. 

For  $f \in S_\kappa^{\dagger, N_0}$, we will show that $\delta([\langle f, z^{-1} e_{\kappa^{-1}} \rangle]) \neq 0.$ We can compute a representing cocycle explicitly by the above recipe: Using (\ref{eq.action-formula-for-sections}), we find
\begin{align} \begin{bmatrix}1 & u \\ 0 & 1\end{bmatrix} \cdot \langle f, z^{-1} e_{\kappa^{-1}} \rangle & = \langle f, (1+uz)\kappa^{-1}(1 + uz) z^{-1} e_{\kappa^{-1}} \rangle \\
&= (z^{-1} + u(1-\Lie \kappa) + \ldots ) \langle  f, e_{\kappa^{-1}} \rangle
\end{align} 
where $\ldots$ is divisible by $z$. 
Subtracting off $\langle f, z^{-1} e_{\kappa^{-1}} \rangle$, we are left with
\[ (u(1-\Lie \kappa) + \ldots ) \langle  f, e_{\kappa^{-1}} \rangle. \]
This is an element of $H^{0,\dagger}(X_{\{0\}}, \calI)$, so 
\[ a_{\begin{bmatrix} 1 & u \\ 0 & 1 \end{bmatrix}}=(u(1-\Lie \kappa) + \ldots ) \langle  f, e_{\kappa^{-1}} \rangle \textrm{ and } b_{\begin{bmatrix} 1 & u \\ 0 & 1 \end{bmatrix}}=0. \]
Thus, $\delta([\langle f, z^{-1} e_{\kappa^{-1}} \rangle]) \in H^1(N_0, H^{0, \dagger}(X_{\{0\}}, \calI))$ is represented by the cocycle 
\[ \begin{bmatrix} 1 & u \\ 0 & 1 \end{bmatrix} \mapsto (u(1-\Lie \kappa) + \ldots ) \langle  f, e_{\kappa^{-1}} \rangle. \]

To verify this cocycle represents a non-zero class, we are free to restrict to any locus inside the canonical ordinary locus $X_{\{0\}}^\circ$. We will thus check that it is non-zero after restriction to a rational open in a standard perfectoid torus inside the ordinary locus, where it will follow from a classical computation in $p$-adic Hodge theory. Because $z=0$ here, we note that the cocycle simplifies to 
\begin{equation}\label{eq:restricted-cocycle} \begin{bmatrix} 1 & u \\ 0 & 1 \end{bmatrix} \mapsto u(1 - \Lie \kappa) \langle f, e_{\kappa^{-1}} \rangle. \end{equation}

Fix an ordinary elliptic curve $E/\overline{\bbF_p}$, then choose a trivialization of the \'{e}tale part of the Tate module of $E$ in order to obtain a Serre-Tate coordinate $q$ on the formal deformation space. The generic fiber of the formal deformation space is the open rigid analytic disk $D: |q-1|<1$, and, after choosing a trivialization of $\bbZ_p(1)$, we obtain canonical $N_0$ level structure on the universal deformation $E_\univ/D$  -- that is, we have 
\[ 0 \rightarrow \bbZ_p \rightarrow T_p E_\univ \rightarrow \bbZ_p \rightarrow 0 \]
where the first $\bbZ_p$ spans the canonical subgroup. By a standard argument (cf., e.g., \cite[7.2]{howe:circle-action} for the same construction on the Igusa formal scheme), splitting this extension can be accomplished by passing to the open perfectoid disk \[ \tilde{D} \sim \lim D \xleftarrow{ q \mapsto q^p} D \xleftarrow{q \mapsto q^p} D \ldots. \]
Indeed, already at the level of formal schemes the map $q \mapsto q^{p^n}$ sends an extension of $p$-divisible groups
\[ 1 \rightarrow \mu_{p^\infty} \rightarrow G \rightarrow \bbQ_p/\bbZ_p \rightarrow 1 \]
to 
\[ 1 \rightarrow \mu_{p^\infty}/\mu_{p^n}=\mu_{p^\infty} \rightarrow G/\mu_{p^n} \rightarrow \bbQ_p/\bbZ_p \rightarrow 1, \]
and this factors through an isomorphism with the degree $p^n$ cover of $D$ parameterizing splittings of the $p^n$-torsion extension 
\[ 1 \rightarrow \mu_{p^n} \rightarrow G[p^n] \rightarrow \frac{1}{p^n}\bbZ_p/\bbZ_p \rightarrow 1 \]
because 
\[ G[p^n]/\mu_p^n \hookrightarrow (G/\mu_{p^n})[p^n] \]
projects isomorphically onto $\frac{1}{p^n}\bbZ_p/\bbZ_p$. This space of splittings is a torsor for $\mu_{p^n} = \Hom(\frac{1}{p^n}\bbZ_p/\bbZ_p, \mu_{p^n})$ and, under the identification of this cover with $q \mapsto q^{p^n}$, the covering action is just by multiplication of the coordinate $q$. 

We thus obtain a map $\tilde{D} \rightarrow X$. Now, we can hit any component of $X_{\{0\}}^\circ$ by changing our trivialization of $\bbZ_p(1)$, and because the overconvergent modular form $f$ comes from finite level, we can choose this trivialization so that the pullback of $f$ does not vanish on $\tilde{D}$ (otherwise there would be a finite level modular curve where $f$ vanished on an open subset of each connected component of the canonical ordinary locus, thus $f$ would be zero). Moreover, one finds that the action of $N_0$ on $X$ is identified with the natural action of $\bbZ_p(1)$ on $\tilde{D}$ -- this is the infinite level consequence of the interpretation of the covering action at finite level in terms of splittings discussed above. 

Now, by (\ref{eq:can-base-transf}), the restriction to $X_{\{0\}}$ of $\langle f, e_{\kappa^{-1}} \rangle$ is $N_0$-invariant (because $z=0$ here). We further restrict to the affinoid perfectoid $\tilde{D}_{|q-1| \leq |p|}$, with ring of functions $A$, and write $g \in A$ for the non-zero and $\bbZ_p(1)$-invariant restriction of $\langle f, e_{\kappa^{-1}} \rangle$. Our restricted cocycle is then 
\[ \begin{bmatrix}1 & u\\0 & 1\end{bmatrix} \mapsto u(1-\Lie\kappa)g. \]
In particular, the class it represents in group cohomology is the image of $(1-\Lie \kappa)g$ under the composition
\[ A^{\bbZ_p(1)} \xrightarrow{\sim} H^1(\bbZ_p(1), A^{\bbZ_p(1)}) \rightarrow H^1(\bbZ_p(1), A), \]
where the first isomorphism is given by evaluating at our choice of generator for $\bbZ_p(1)$. It is a standard result that the second arrow is also an isomorphism (cf. \cite[Lemmas 5.5 and 6.18]{scholze:p-adic-ht}), thus, because $\Lie\kappa \neq 1$ and $g \neq 0$, we conclude the cohomology class is non-zero. 

\end{proof}

\begin{remark} 
More canonically, the cocycle after restriction to $\tilde{D}$ in the proof above is identified with the differential form $\langle f, e_{\kappa^{-1}} \rangle \frac{dq}{q}.$ In fact, it should be possible to make this identification over the generic fiber of the entire Igusa formal scheme (with $\frac{dq}{q}$ replaced by the differential form identified with $e_{z^2}$ via Kodaira-Spencer). One runs into a delicacy here because the generic fiber of the Igusa formal scheme does not fit directly into the framework of \cite{scholze:p-adic-ht}; however, it should be possible to directly compute with the Faltings extension on the perfectoid Igusa tower (as opposed to standard torus coordinates) to obtain this identification. 
\end{remark}

This concludes the proof of Theorem \ref{maintheorem:ocmf-cup} when $\Lie \kappa \not\in \bbZ_{\geq 2}$, as in this case $V_{\kappa^{-1}}$ is irreducible.  For $\kappa=z^{k}$, $k \geq 2$, we can only conclude that the kernel is of the form $W \otimes V'_{-k}$ for some subspace $W \subset S_k^\dagger$: Indeed, any $\gl_2$-submodule $M \subset S_k^\dagger \otimes V_{-k}$ will be generated by its highest weight vectors, but if we write $U$ for the one-dimensional highest weight space generating $V_{-k}$ and $U'$ for the one dimensional highest weight space generating $V_{-k}'$, then the highest weight vectors in $M$ are given by 
\[ M \cap (S_k^\dagger \otimes U) \oplus M \cap (S_k^\dagger \otimes U'),\]
and for $M$ the kernel we have just shown the first summand is zero, thus the second gives the subspace $W$.

We obtain an induced injection on 
\[ W \otimes V_{-k}/V_{-k}'= W \otimes H^1(\bbP^1, \calO(-k)). \]
In \S\ref{ss:local-cup-compat} we saw that $S_k^\cl \subset W$, and that the induced map on $S_k^\cl \otimes H^1(\bbP^1, \calO(-k))$ is identified with the global cup product. Thus, it remains only to show that $W$ is no larger than $S_k^\cl$. 

To show this, we observe that the cup product is already defined over $\bbQ_p$, and thus the image lands in the locally algebraic vectors in the Hodge-Tate weight zero part of $\tilde{H}^1_c$. On the other hand, Lemma \ref{lemma:loc-alg-ht-weight-zero} below implies that taking locally algebraic vectors commutes with passage to the Hodge-Tate weight zero part, thus the computation in \ref{sss:intro-loc-alg} shows that the locally algebraic vectors of type $H^1(\bbP^1, \calO(-k))$ in the Hodge-Tate weight zero part are \emph{abstractly} isomorphic to $S_k^\cl \otimes H^1(\bbP^1, \calO(-k))$. Since the map on $W \otimes H^1(\bbP^1, \calO(-k))$ is an injection and $S_k^\cl \subset W$, admissibility of $S_k^\cl$ implies that $W$ must in fact be equal to $S_k^\cl$. 

\begin{remark}
Using a $K_p$-equivariant Mayer-Vietoris sequence to compute 
\[ H^1(X, \Sym^k \calO_X^2 \otimes \calI)=H^1(X, \calI) \otimes \Sym^k \bbC_p^2, \]
then taking $K_p$-invariants, one naturally recovers the Hodge-Tate filtration 
\[ 0\rightarrow H^1(X/K_p, \omega^{-k} \otimes \calI)(k) \rightarrow \Hom_{K_p}( (\Sym^k \bbC_p^2)^*, \tilde{H}^1_{c,\bbC_p} )  \rightarrow S_{k+2}^{K_p}(1) \rightarrow 0\]
essentially by Falting's \cite{faltings:hodge-tate} method (the key point is that if we pass to sheaves of smooth vectors in (\ref{eq:hodge-tate-ext}) then the induced boundary map is naturally identified with the Kodaira-Spencer isomorphism). It follows from a result of Emerton (see Theorem \ref{emerton:loc-alg} below) that the natural map 
\[ H^1_c(Y_{K_pK^P}, \Sym^k \bbQ_p^2) \rightarrow \Hom_{K_p}((\Sym^k\bbQ_p^2)^*, \tilde{H}^1_c) \]
is an isomorphism, and we recover Falting's Hodge-Tate decomposition. Using this interpretation, one can compute that our global cup product is identified with $(k-1)$ times the canonical splitting of the Hodge-Tate filtration composed with Emerton's isomorphism. The dual cup product (\ref{eq:second-half-cup}) is identified on the nose with the inclusion in the Hodge-Tate filtration composed with Emerton's isomorphism. 
\end{remark}

\section{The Hodge-Tate weight zero part of completed cohomology}\newcommand{\ohat}{\widehat{\otimes}}
\label{sec:hodge-tate}

In this section we recall more carefully the relation between completed cohomology and $H^1(X, \calI)$, and prove Corollary \ref{c:hts-weights}. The main point is to verify that passing to the Hodge-Tate weight zero subspace commutes with natural representation theoretic operations (locally algebraic vectors, Hecke eigenspaces). This is carried out in a few simple lemmas, and then Corollary \ref{c:hts-weights} is a straightforward consequence of a weak form of local-global compatibility for completed cohomology and the Galois equivariance of our cup-product constructions. 

\subsection{First lemmas}

\begin{definition}
If $E \subset \bbC_p$ is a finite extension of $\bbQ_p$ and $V$ is a unitary $E$-Banach representation of $G_{E}$, we write 
\[ \HT^E_0(V) = (V \widehat{\otimes}_{E} \bbC_p)^{G_{E}}. \]
where here we take the semilinear action on $V \ohat_E \bbC_p$. When $E$ is apparent from the context, we will drop the superscript. 
\end{definition}

Recall that any (continuous) finite dimensional representation of $G_E$ on an $E$-vector space can be equipped with the structure of a unitary representation (i.e. fixes an $\calO_E$-lattice). Any two lattices induce equivalent norms, thus we can work with finite dimensional representations without fixing any extra information. 

If $V$ is finite dimensional, then the Hodge-Tate weight zero part of $V$ in the classical sense is $\HT^E_0(V)\otimes_E \bbC_p \subset V \otimes_E \bbC_p$. For our purposes we will work only with the invariants $\HT^E_0(V)$, however.

\begin{lemma}\label{lemma:set-of-operators} Suppose $V$ and $W$ are unitary $E$-Banach representations of $G_{E}$ and $S$ is a collection of bounded $G_E$-equivariant operators $V \rightarrow W$. We write $V_S$ for their simultaneous kernel, a closed subspace of $V$. Then 
\[ \HT^E_0(V_S)=\HT^E_0(V)_S.\] 
\end{lemma}
\begin{proof} After possibly replacing the norm on $V$ with an equivalent norm, we may choose an orthonormal basis for $V_S$, then extend it to an orthonormal basis for $V$. Using this basis, it is clear that $V_S \ohat \bbC_p = (V \ohat \bbC_p)_S$. Taking $G_E$-invariants commutes with passage to the kernel of the operators in $S$, so we obtain 
\[  (V_S \ohat \bbC_p)^{G_E} = \left((V \ohat \bbC_p)^{G_E}\right)_S, \]
and we conclude. 
\end{proof}

\begin{lemma}\label{lemma:triv-ht-split} If $V$ is an $E$-Banach space with the trivial $G_E$-action and $W$ is an $E$-banach space with a unitary $G_E$-action then 
\[ \HT^E_0(V \widehat{\otimes} W)=V \widehat{\otimes} \HT^E_0(W). \]
In particular, if one of $V$ or $W$ is finite dimensional
\[    \HT^E_0(V \otimes W)=V \otimes \HT^E_0(W). \]
\end{lemma} 
\begin{proof}
After passing to an equivalent norm on $V$ we may choose an orthonormal $E$-basis $\{e_i\}_{i \in I}$. Then 
\[ V \widehat{\otimes} W = \hat{\oplus}_{i \in I} \left(e_i \otimes W \right) \]
is an orthonormal decomposition. Thus, any element of $(V \widehat{\otimes} W) \widehat{\otimes} \bbC_p$ has a unique expression as $( e_i \otimes w_i)_i $
for $\{w_i\}_{i \in I}$ a collection of vectors in $W \widehat{\otimes} \bbC_p$ such that for any $\epsilon>0$ there is a finite set $J \subset I$ such that $||w_i|| < \epsilon$ for all $i \in I\bs J$. The $G_{E}$-invariants are then precisely those vectors with 
\[ w_i \in (W\wh{\otimes}\bbC_p)^{G_{E}} = \HT^E_0(W) \textrm{ for all $i\in I$}\]
 and this is exactly $V\wh{\otimes} \HT^E_0(W).$   
\end{proof}

\subsection{Compactly supported completed cohomology}\newcommand{\Rlim}{\mathrm{Rlim}}
Recall from \cite{emerton:on-the-interpolation} that the (degree one) compactly supported completed cohomology of the modular curve with $\bbQ_p$-coefficients is defined as
\[ \tilde{H}^1_c:=\left(\lim_n \colim_{K_p} H^i_c(Y_{K_pK^p}, \bbZ/p^n)\right) [1/p], \]
where the $Y_{K_pK^p}$ are finite level (open) modular curves. Completed cohomology with $\bbC_p$-coefficients $\tilde{H}^1_{c,\bbC_p}$ is obtained by replacing $\bbZ/p^n$ with $\calO_{\bbC_p}/p^n$ in the definition.  Both are equipped with their obvious Banach topologies, and $\tilde{H}^1_{c,\bbC_p} = \tilde{H}^1_{c} \widehat{\otimes} \bbC_p.$  Before explaining the identification of $\tilde{H}^1_{c, \bbC_p}$ with the analytic cohomology $H^1(X, \calI)$, we recall a result of Emerton that we will need later on:
\begin{theorem}[Emerton]\label{emerton:loc-alg} For $W$ an an algebraic representation of $\GL_2/\bbQ_p$ and $K_p \subset \GL_2(\bbQ_p)$ a compact open subgroup, there is a canonical identification
\[ \Hom_{\bbQ_p[K_p]}(W, \tilde{H}^1_c) = \Hom_{\bbQ_p[K_p]}(W, \tilde{H}^{1, \locan}_c) =H^i(Y_{K_pK^p}, \calV_{W^*}) \]
where $\calV_{W^*}$ denotes the natural $\bbQ_p$ local system on $Y_{K_pK^p}$ constructed from the dual representation $W^*$. In particular, it is a finite dimensional $\bbQ_p$-vector space.  	
\end{theorem}
\begin{proof}
The identity of the first two terms is immediate because $K_p$-algebraic vectors are locally analytic. The second identity
follows from taking $K_p$-invariants in equation (4.3.7) on p.63 in \cite{emerton:on-the-interpolation} (note that $\hat{H}^1_c=\tilde{H}^1_c$ in this case, as explained at the top of p.58 of loc. cit.). 
\end{proof}

\begin{lemma}\label{lemma:comparison} There is a natural $G_{\bbQ_p}$-equivariant isomorphism ${H^1(X, \calI) \cong \tilde{H}^1_{c, \bbC_p}}$, where the right-hand side is equipped with the semilinear action. \end{lemma}
\begin{proof}
In \cite[Theorem 4.2.1]{scholze:torsion}, it is shown that the map $j_! \bbZ/p^n \rightarrow \calI^+/p^n$ (for $j$ the immersion of the open modular curve into the compactified modular curve) induces an almost isomorphism
\begin{equation}\label{eq:scholze-torsion-comparison} H^i(X, \calI^+/p^n) \cong_a \tilde{H}^i_c(\bbZ/p^n) \otimes \calO_{\bbC_p}/p^n. \end{equation} 
It remains to show this passes to the limit: We write $\partial X$ for the boundary. Because $\partial X$ is strongly Zariski closed, we have for each $n$ an (almost) exact sequence 
\[ 0 \rightarrow \calI^+/p^n \rightarrow \calO_X^+/p^n \rightarrow \calO^+_{\partial X} / p^n \rightarrow 0. \]
By the almost version of \cite[Lemma 3.18]{scholze:p-adic-ht}, using a basis of affinoid perfectoids we find that the second two terms have almost vanishing $R^i \lim$ for $i \geq 1$. Thus, taking the long-exact sequence and using almost surjectivity of $\calO_X^+ \rightarrow \calO^+_{\partial X}$, we conclude that so does the system $\calI^+/p^n$, so that 
\begin{equation}\label{eq:rlim-almost-lim} \Rlim \calI^+/p^n =_a \lim \calI^+/p^n = \calI^+. \end{equation}

By \cite[0D6K]{stacks-project}, we have an exact sequence
\[ 0 \rightarrow R^1 \lim H^0(X, \calI^+/p^n) \rightarrow H^1(X, \Rlim \calI^+/p^n) \rightarrow \lim H^1(\calI^+/p^n) \rightarrow 0. \]
But by the $i=0$ case of (\ref{eq:scholze-torsion-comparison}) and vanishing of compactly supported $H^0$, the first term is almost zero, and by the $i=1$ case the right is almost equal to $\tilde{H}^1_{c, \calO_{\bbC_p}}$. By (\ref{eq:rlim-almost-lim}), the middle term is almost $H^1(X, \calI^+)$, and finally, inverting $p$, we obtain the desired isomorphism. It is Galois-equivariant because it comes from the map $j_! \bbZ/p^n \rightarrow \calI^+/p^n$. 
\end{proof}

We now verify that taking locally algebraic vectors commutes with passing to the Hodge-Tate weight zero part, which was used in the proof of Theorem \ref{maintheorem:ocmf-cup}.

\begin{lemma}\label{lemma:loc-alg-ht-weight-zero}
The natural inclusion 
\[ (\tilde{H}^{1, \localg}_{c} \otimes \bbC_p)^{G_{\bbQ_p}} \rightarrow \left(\HT^{\bbQ_p}_0(\tilde{H}^{1}_{c})\right)^\localg \]
is an isomorphism. 
\end{lemma}
\begin{proof}
It suffices to show that for any choice of an irreducible algebraic representation $W$ of $\GL_2/\bbQ_p$ and a compact open subgroup $K\subset \GL_2(\bbQ_p)$, the statement holds after replacing locally algebraic vectors everywhere with the $W$-isotypic part for the $K$-action. Expressing these vectors via the standard evaluation maps, this means we need to show that
\[ \left( \left(W\otimes\Hom_{\bbQ_p[K]}(W, \tilde{H}^{1}_{c})\right) \otimes \bbC_p\right)^{G_{\bbQ_p}} = W \otimes \Hom_{\bbQ_p[K]}(W, \HT_0( \tilde{H}^{1}_{c})). \]
To see this, we first apply Lemma \ref{lemma:set-of-operators} to 
$\Hom_{\bbQ_p}(W, \tilde{H}^1_c)$  with $S$ the set of operators $(k - \Id)$ for $k \in K$ to obtain
\begin{align*}
	\HT_0\left(\Hom_{\bbQ_p[K]}(W, \tilde{H}^{1}_{c})  \right) &= \HT_0\left( \left(\Hom_{\bbQ_p}(W, \tilde{H}^{1}_{c})\right)_S \right)\\
	&= \left(\HT_0(\Hom_{\bbQ_p}(W, \tilde{H}^{1}_{c}))\right)_S \\
	&= \left(\HT_0(W^*\otimes \tilde{H}^{1}_{c})\right)_S \\
	&= \left(W^* \otimes \HT_0( \tilde{H}^{1}_{c})\right)_S \\
	&= \left(\Hom_{\bbQ_p}(W, \HT_0(\tilde{H}^{1}_{c}))\right)_S \\
	&= \Hom_{\bbQ_p[K]}(W, \HT_0( \tilde{H}^{1}_{c})).
\end{align*}
Here to pass from the third line to the fourth we have used Lemma \ref{lemma:triv-ht-split}. Tensoring with $W$ and applying Lemma \ref{lemma:triv-ht-split} again we obtain
\[ \HT_0\left(W\otimes \Hom_{\bbQ_p[K]}(W, \tilde{H}^{1}_{c})  \right) = W \otimes \Hom_{\bbQ_p[K]}(W, \HT_0( \tilde{H}^{1}_{c})). \]
By Theorem \ref{emerton:loc-alg}, $\Hom_{\bbQ_p[K]}(W, \tilde{H}^1_c)$ is finite dimensional, so the completed tensor product in the formation of $\HT_0$ on the left is just a tensor product, and thus the left-hand side is equal to
\[ \left( \left(W\otimes\Hom_{\bbQ_p[K]}(W, \tilde{H}^{1}_{c})\right) \otimes \bbC_p\right)^{G_{\bbQ_p}}. \]
\end{proof}

\subsection{Proof of Corollary \ref{c:hts-weights}}\label{subsec:galois-corollary-proof}\newcommand{\bbT}{\mathbb{T}}

As always we have fixed a prime-to-$p$ level $K^p$. For $\Sigma$ a finite set of primes containing $p$ and those ramified in $K^p$, let $\bbT=\bbT_{\Sigma}$ be the tame Hecke algebra generated by the spherical Hecke operators at primes $\ell \not\in \Sigma$. Let $E \subset \bbC_p$ be a finite extension of $\bbQ_p$ and $f$ a $p$-adic modular form of level $K^p$ defined over $E$. If $f$ is a $\bbT$-eigenform, there is an associated maximal ideal $\frakp$ of $\bbT[1/p]$ and, after possibly enlarging $E$, a semisimple 2-dimensional representation $\rho=\rho(f)$ of $G_{\bbQ}$ over $E$. 

Assume that $\overline{\rho}$ is absolutely irreducible. As a consequence of \cite[Lemma 5.5.3]{emerton:local-global-proof}, there is an $E$-Banach space $V$ on which $G_{\bbQ_p}$ acts trivially such that
\begin{equation}\label{eqn:lgc} \tilde{H}^1_{c,E}[\frakp] = \tilde{H}^1_{E}[\frakp] = V \otimes \rho. \end{equation}
Here for the first equality we use that the kernel of the surjective map $\tilde{H}^1_{c,E} \rightarrow \tilde{H}^1_E$ comes from the zero dimensional cohomology of the boundary -- this cohomology is Eisenstein, so, by the condition on $\overline{\rho}$, we can localize away from it. 

Suppose now in addition that $f$ is cuspidal overconvergent and of weight $\kappa$ with $\Lie \kappa \neq 1$. Applying Lemmas \ref{lemma:set-of-operators} and \ref{lemma:triv-ht-split}, we deduce that 
\[ \HT_0^E(\tilde{H}^1_{c,E})[\frakp] = V \otimes \HT_0^E(\rho). \]
Because the pairing of Theorem \ref{maintheorem:ocmf-cup} can be constructed already over $E$, $f$ gives rise to non-zero elements in the left-hand side. Thus, $\HT_0^E(\rho) \neq 0$, so $0$ is a Hodge-Tate-Sen weight of $\rho$. The determinant of the Galois representation attached to \emph{any} $p$-adic modular form of weight $\kappa$ is a character with infinitesimal weight $\Lie\kappa -1$, thus we conclude that the Hodge-Tate-Sen weights of $\rho$ are $0$ and $\Lie \kappa - 1$. 

\begin{remark} Here we are taking the $\sigma$-Hodge-Tate-Sen weights with respect to $\sigma$ our fixed embedding of $E$ in $\bbC_p$, but it follows from this result that for any embedding $\sigma:E \rightarrow \bbC_p$ the $\sigma$-Hodge-Tate-Sen weights of $\rho$ are $0$  and $\sigma(\Lie \kappa) - 1.$ 
\end{remark}

\bibliographystyle{plain}
\bibliography{refs}

\end{document}

%% file: macros.tex
\renewcommand{\det}{\mathrm{det}}

\newcommand{\Id}{\mathrm{Id}}

\newcommand{\Spa}{\mathrm{Spa}}

\newcommand{\Hom}{\mathrm{Hom}}

\newcommand{\Lie}{\mathrm{Lie}}

\newcommand{\GL}{\mathrm{GL}}

\newcommand{\et}{\mathrm{\acute{e}t}}

\newcommand{\cl}{\mathrm{cl}}

\newcommand{\Sym}{\mathrm{Sym}}

\newcommand{\HT}{\mathrm{HT}}

\newcommand{\sm}{\mathrm{sm}}

\newcommand{\univ}{\mathrm{univ}}

%% file: theoremnumbering.tex
\numberwithin{equation}{subsubsection}
\theoremstyle{plain}

\newtheorem{maintheorem}{Theorem} 
 
\newtheorem{maincorollary}[maintheorem]{Corollary}

\newtheorem*{theorem*}{Theorem}

\newtheorem{theorem}[subsubsection]{Theorem}

\newtheorem{lemma}[subsubsection]{Lemma}

\theoremstyle{definition}

\newtheorem{definition}[subsubsection]{Definition}
\newtheorem{remark}[subsubsection]{Remark}